\edef\savecatcodeat{\the\catcode`@}
\catcode`\@=11

\def\tb@ifSpecChars#1#2{#1}
\def\tb@ifNoSpecChars#1#2{#2}

\def\tableau{%
  \bgroup
  \@ifstar{\let\Tif\tb@ifNoSpecChars\tb@tableauB}
          {\let\Tif\tb@ifSpecChars\tb@tableauB}}

\def\tb@tableauB{
  \@ifnextchar[{\tb@tableauC}{\tb@tableauC[]}}

\def\tb@tableauC[#1]{\hbox\bgroup%
    \let\\=\cr
    \def\bl{\global\let\tbcellF\tb@cellNF}%
    \def\tf{\global\let\tbcellF\tb@cellH}
    \dimen2=\ht\strutbox \advance\dimen2 by\dp\strutbox%
    \ifx\baselinestretch\undefined\relax%
    \else%
       \dimen0=100sp \dimen0=\baselinestretch\dimen0%
       \dimen2=100\dimen2 \divide\dimen2 by\dimen0%
    \fi%
    \let\tpos\tb@vcenter
    \tb@initYoung
    \tb@options#1\eoo
    \let\arrow\tb@arrow%
    \dimen0=\Tscale\dimen2%
    \dimen1=\dimen0 \advance\dimen1 by \tb@fframe%
    \lineskip=0pt\baselineskip=0pt
    \def\tb@nothing{}%
    \def\endcellno{$\rss\egroup\bss\egroup}
    \def\endcell{\endcellno\kern-\dimen0}
    \def\begincell{\vbox to\dimen0\bgroup\vss\hbox to\dimen0\bgroup\hss$}%
    \let\overlay\tb@overlay%
    \let\fl\tb@fl%
    \let\lss\hss\let\rss\hss\let\tss\vss\let\bss\vss
    \def\mkcell##1{
        \let\tbcellF\tb@cellD
        \def\tb@cellarg{##1}
        \ifx\tb@cellarg\tb@nothing\let\tb@cellarg\tb@cellE\fi%
        \begincell\tb@cellarg\endcellno
        \tbcellF}
    \let\savecellF\tbcellF
     \Tif{\catcode`,=4\catcode`|=\active}{}\tb@tableauD}%

\let\tb@savetableauD\tableauD
{
    \catcode`|=\active \catcode`*=\active \catcode`~=\active%
    \catcode`@=\active
\gdef\tableauD#1{%
  \Tif{
    \mathcode`|="8000 \mathcode`*="8000%
    \mathcode`~="8000 \mathcode`@="8000%
    \def@{\bullet}%
    \let|\cr
    \let*\tf
    \let~\sk
  }{}%
  \tpos{\tabskip=0pt\halign{&\mkcell{##}\cr#1\crcr}}%
  \global\let\tbcellF\savecellF
  \egroup
  \egroup}
}
\let\tb@tableauD\tableauD
\let\tableauD\tb@savetableauD
\let\tb@savetableauD\undefined

\def\tb@options#1{\ifx#1\eoo\relax\else\tb@option#1\expandafter\tb@options\fi}

\def\tb@option#1{%
  \if#1t\let\tpos\tb@vtop\fi
  \if#1c\let\tpos\tb@vcenter\fi
  \if#1b\let\tpos\vbox\fi
  \if#1F\tb@initFerrers\fi
  \if#1Y\tb@initYoung\fi
  \if#1s\tb@initSmall\fi
  \if#1m\tb@initMedium\fi
  \if#1l\tb@initLarge\fi
  \if#1p\tb@initPartition\fi
  \if#1a\tb@initArrow\fi
}

\def\tb@vcenter#1{\ifmmode\vcenter{#1}\else$\vcenter{#1}$\fi}

\def\tb@vtop#1{\hbox{\raise\ht\strutbox\hbox{\lower\dimen0\vtop{#1}}}}

\def\tb@initPartition{\def\Tscale{.3}}
\def\tb@initSmall{\def\Tscale{1}}
\def\tb@initMedium{\def\Tscale{2}}
\def\tb@initLarge{\def\Tscale{3}}

\def\tb@initArrow{\dimen2=1.25em}

\def\tb@initYoung{%
  \def\tb@cellE{}
  \let\tb@cellD\tb@cellN
  \def\sk{\global\let\tbcellF\tb@cellNF}}
\def\tb@initFerrers{%
  \def\tb@cellE{\bullet}
  \let\tb@cellD\tb@cellNF
  \def\sk{\bullet}}

\tb@initMedium

\def\tb@sframe#1{%
  \vbox to0pt{
    \vss
    \hbox to0pt{%
      \hss
      \vbox to\dimen1{
        \hrule depth #1 height0pt
        \vss
        \hbox to\dimen1{
          \vrule width #1 height\dimen1
          \hss
          \vrule width #1
          }%
        \vss
        \hrule height #1 depth 0in
        }%
      \kern-\tb@hframe
      }%
    \kern-\tb@hframe}}

\def\tb@hframe{.2pt}\def\tb@fframe{.4pt}\def\tb@bframe{2pt}
\def\tb@cellH{\tb@sframe{\tb@bframe}}       
\def\tb@cellNF{}                            
\def\tb@cellN{\tb@sframe{\tb@fframe}}       
\let\tbcellF\tb@cellN                       

\def\tb@rpad{1pt}
\def\tb@lpad{1pt}
\def\tb@tpad{1.8pt}
\def\tb@bpad{1.8pt}

\def\tb@overlay{\endcell\@ifnextchar[{\tb@overlaya}{\begincell}}
\def\tb@overlaya[#1]{\vbox to\dimen0\bgroup%
  \tb@overlayoptions#1\eoo%
  \tss\hbox to\dimen0\bgroup\lss$}
\def\tb@overlayoptions#1{\ifx#1\eoo\relax\else\tb@overlayoption#1\expandafter\tb@overlayoptions\fi}

\def\tb@overlayoption#1{
  \if#1t\def\tss{\vskip\tb@tpad}\let\bss\vss\fi
  \if#1c\let\tss\vss\let\bss\vss\fi
  \if#1b\def\bss{\vskip\tb@bpad}\let\tss\vss\fi
  \if#1l\def\lss{\hskip\tb@lpad}\let\rss\hss\fi
  \if#1m\let\lss\hss\let\rss\hss\fi
  \if#1r\def\rss{\hskip\tb@rpad}\let\lss\hss\fi
}

\def\tb@fl{\endcell\begincell\vrule depth 0pt width \dimen0 height \dimen0 \endcell\begincell}

\@ifundefined{diagram}{}{
\def\tb@arrowpad{.5}

\newoptcommand{\tb@arrow}{\@ne}[2]{%
  \endcell
   \begingroup%
   \let\dg@getnodesize\tb@getnodesize
   \dg@USERSIZE=#1\relax%
   \ifnum\dg@USERSIZE<\@ne \dg@USERSIZE=\@ne \fi%
   \dg@parse{#2}%
   \dg@label{\tb@draw{#1}{#2}}}

\def\tb@getnodesize#1#2#3#4#5{\dimen3=\tb@arrowpad\dimen2 #4=\dimen3 #5=\dimen3\relax}
\def\tb@getnodesize#1#2#3#4#5{\ifnum#2=0\ifnum#3=0\tb@getnodesizetail{#4}{#5}\else\tb@getnodesizehead{#4}{#5}\fi\else\tb@getnodesizehead{#4}{#5}\fi}
\def\tb@getnodesizetail#1#2{\dimen3=.5\dimen2 #1=\dimen3 #2=\dimen3}
\def\tb@getnodesizehead#1#2{\dimen3=.5\dimen2 #1=\dimen3 #2=\dimen3}

\def\tb@draw#1#2#3#4{%
        \dg@X=0\dg@Y=0\dg@XGRID=1\dg@YGRID=1\unitlength=.001\dimen0%
        \dg@LBLOFF=\dgLABELOFFSET \divide\dg@LBLOFF\unitlength%
        \dg@drawcalc
        \begincell
        \let\lams@arrow\tb@lams@arrow
        \begin{picture}(0,0)\begingroup\dg@draw{#1}{#2}{#3}{#4}\end{picture}%
        \endcell
        \endgroup
        \begincell}
}

\def\tb@lams@arrow#1#2{%
 \lams@firstx\z@\lams@firsty\z@
 \lams@lastx#1\relax\lams@lasty#2\relax
 \lams@center\z@
 \N@false\E@false\H@false\V@false
 \ifdim\lams@lastx>\z@\E@true\fi
 \ifdim\lams@lastx=\z@\V@true\fi
 \ifdim\lams@lasty>\z@\N@true\fi
 \ifdim\lams@lasty=\z@\H@true\fi
 \NESW@false
 \ifN@\ifE@\NESW@true\fi\else\ifE@\else\NESW@true\fi\fi
 \ifH@\else\ifV@\else
  \lams@slope
  \ifnum\lams@tani>\lams@tanii
   \lams@ht\ten@\p@\lams@wd\ten@\p@
   \multiply\lams@wd\lams@tanii\divide\lams@wd\lams@tani
  \else
   \lams@wd\ten@\p@\lams@ht\ten@\p@
   \divide\lams@ht\lams@tanii\multiply\lams@ht\lams@tani
  \fi
 \fi\fi
 \ifH@  \lams@harrow
 \else\ifV@ \lams@varrow
 \else \lams@darrow
 \fi\fi
}

\catcode`\@=\savecatcodeat
\let\savecatcodeat\undefined

\documentclass[tbtags,reqno]{amsart}
\usepackage{graphicx,color}
\usepackage{amsmath,amsfonts,amstext,amsbsy,amsopn,amsthm,amscd,amsxtra,dsfont,amssymb,pb-diagram}
\usepackage{amsmath,amsthm,amsopn,amsfonts,amssymb,epsfig,enumerate}
\usepackage{amsthm,amsopn,amsfonts,amssymb,epsfig,mathrsfs,cancel,wasysym}
\usepackage[all]{xy}
\usepackage{youngtab}
\usepackage{ytableau}
\usepackage[active]{srcltx}
\newcommand{\ysdiagram}[1]{\ytableausetup{smalltableaux,centertableaux}
\ytableaushort{#1}}

\numberwithin{equation}{section}

\makeatletter
\renewcommand{\subsubsection}{\@startsection
{subsubsection} {3} {0mm} {\baselineskip} {-0.5\baselineskip} {\normalfont\normalsize\bfseries}} \makeatother

\newtheorem{theorem}{Theorem}
\newtheorem{lemma}[theorem]{Lemma}
\newtheorem{proposition}[theorem]{Proposition}
\newtheorem{example}[theorem]{Example}

\newtheorem{corollary}[theorem]{Corollary}

\newtheorem{definition}[theorem]{Definition}

\newtheorem{remark}[theorem]{Remark}
\newtheorem*{acknow}{Acknowledgments}

\def\la{{\lambda}}

\def\cal L{{\mathcal L}}

\newcommand{\La}{\Lambda}

\newcommand{\ta}{\theta}

\def\la{{\lambda}}

\def\cal L{{\mathcal L}}

\def\cd{\circledast}
\def\deg{ {\rm {deg}}}

\def\La{{ \Lambda}}

\let\y\infty

\def\cd{{\circledast}}

\def\B{{\mathcal B}}

\let\la\lambda
\let\La\Lambda

\let\Om\Omega

\let\ta\theta

\newcommand{\LL}{\ensuremath{\langle\!\langle}}
\newcommand{\RR}{\ensuremath{\rangle\!\rangle}}

\def \part {\vdash}
\def\La {\Lambda}

\def \Om {\Omega}
\def \ta {\theta}
\def\cd{\circledast}

\newcommand{\tcercle}[1]{\ensuremath{\setlength{\unitlength}{1ex}\begin{picture}(2.8,2.8)\put(1.4,1.4){\circle{2.8}\makebox(-5.6,0){#1}}\end{picture}}}

\newcommand{\btcercle}[1]{\ensuremath{\setlength{\unitlength}{1ex}\begin{picture}(12.8,12.8)\put(6.8,0.4){\circle{5.2}\makebox(0,0){#1}}\end{picture}}}

\def\cd{{\circledast}}

\def\B{\mathcal{B}}

\def\S{\mathcal{S}}

\def\beq{\begin{equation}}
\def\eeq{\end{equation}}
\def\inv{\text{inv}}

\begin{document}

\title[The norm and the evaluation of the Macdonald polynomials in superspace]
{The norm and the evaluation of the Macdonald polynomials in superspace}

\author{Camilo Gonz\'alez}
\address{Instituto de Matem\'atica y F\'{\i}sica, Universidad de Talca, Casilla 747, Talca,
Chile} \email{cgonzalez@inst-mat.utalca.cl}

\author{Luc Lapointe}
\address{Instituto de Matem\'atica y F\'{\i}sica, Universidad de Talca, Casilla 747, Talca,
Chile} \email{lapointe@inst-mat.utalca.cl}


\begin{abstract}
  We demonstrate the validity of previously conjectured 
  explicit expressions for the norm and the evaluation of the Macdonald polynomials in superspace.  These expressions, which involve the  arm-lengths and leg-lengths of the cells in certain Young diagrams,  specialize to
    the well known formulas for the norm and the evaluation of the usual Macdonald polynomials.
\end{abstract}

\keywords{Macdonald polynomials, Evaluation, Symmetric functions in superspace}

\maketitle

\section{Introduction}

An extension to superspace of the Macdonald polynomials was presented
in \cite{BDLM,BDLM2}.  In this setting, the polynomials not only depend on the usual commuting variables $x_1,x_2,\dots$ but also on anticommuting variables
$\theta_1,\theta_2,\dots$. Just as is the case for the usual Macdonald polynomials \cite{Mac},
the Macdonald polynomials in superspace can be characterized by conditions of  triangularity and orthogonality (we refer to Section~\ref{sps} for the relevant definitions): 
\begin{theorem}[\cite{BDLM2}]\label{theo1}
Given a superpartition $\La=(\La^a;\La^s)$ of fermionic degree $m$,
there is a unique family of symmetric polynomials in superspace   $\{P_\La=P_\La(x,\ta;q,t)\}_\Lambda$,
with $x=(x_1,x_2,\dots)$ and $\theta=(\theta_1,\theta_2,\dots)$, such that: 
\begin{equation}\label{mac1}
\begin{array}{lll} 1)& P_{\Lambda} =
m_{\Lambda} + \text{lower terms},\\
2)&\LL  P_{\La}| P_{\Om} \RR_{q,t} =0\quad\text{if}\quad \La\ne\Om,
\end{array}\end{equation}
where ``lower terms'' are with respect to the dominance  ordering on superpartitions.
The scalar product is 
defined on the super-power sum basis by 
\begin{equation}\label{newsp}
\LL {p_\La}|{p_\Om}\RR_{q,t}=(-1)^{\binom{m}2}\,z_\La(q,t)\delta_{\La\Omega}\, ,
\end{equation}where\begin{equation}\label{newz}
\qquad z_\La(q,t)
= z_{\La^s} 
\, q^{|\La^a|} \prod_{i=1}^{\ell(\La^s)}
\frac{1-q^{\La^s_i}}{1-t^{\La^s_i}}\, .\, \end{equation}  
\end{theorem}
Most of the important features of the Macdonald polynomials seem to extend
to superspace. The duality and the existence of Macdonald operators were shown for instance in \cite{BDLM2} while various properties were conjectured to hold such as explicit formulas for the norm and the evaluation \cite{BDLM}, Pieri rules (which connect to the 6-vertex model) \cite{GJL} and symmetry \cite{O}. Most remarkably, 
an extension of the original Macdonald positivity conjecture was stated in \cite{BDLM2}.

In this article we prove the conjectures concerning the  norm and the evaluation.
To describe those results, we first need to introduce some notation.
A superpartition $\Lambda$ is in bijection with a pair of partitions $(\Lambda^\cd,\Lambda^*)$, with $\Lambda^* \subseteq \Lambda^\cd$ and such that the skew diagram $\Lambda^\cd/\Lambda^*$ contains at most one box in each row and in each column.  Superpartitions can  naturally be represented by Ferrers diagrams where the cells in $\Lambda^\cd/\Lambda^*$ are drawn as circles.
Consider for instance the superpartition $\Lambda$ given by the partitions $\Lambda^\cd=(6,5,4,2,2,1,1)$ and $\Lambda^*=(6,4,3,2,1,1)$ whose respective Ferrers diagrams are
$$
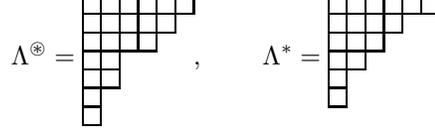
\Lambda^\cd=\tiny{\tableau[scY]{ & & & & & \\   & & & & \\ & & & \\ & \\ &  \\ &\bl \\ &\bl  },\,\,\,\,\,\,\,\,\,\,\,\,\,\Lambda^*=\tableau[scY]{ & & & & & \\   & & & &\bl\\ & & & \bl \\ & \\ & \bl \\ &\bl \\ \bl&\bl  }}
$$
Replacing the cells of $\Lambda^\cd/\Lambda^*$ by circles, the superpartition $\Lambda$ is represented by the diagram
$$
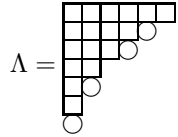
\Lambda=\tiny{\tableau[scY]{ & & & & & \\   & & & &\bl\tcercle{}\\ & & & \bl\tcercle{}\\ & \\ & \bl\tcercle{} \\ &\bl \\ \bl\tcercle{}&\bl  }}
$$

We prove that the norm squared of the Macdonald polynomials in superspace $|| P_\La||^2=\LL  P_{\La}| P_{\La} \RR_{q,t}$ with respect to the scalar product \eqref{newsp}
is such that (see Proposition~\ref{propnorm})
\begin{align} \label{ns}
\|P_\Lambda\|^2=q^{|\Lambda^a|} \prod_{s\in \mathcal{B}\Lambda} \frac{1-q^{a_{\Lambda^*}(s)+1}t^{l_{\Lambda^\circledast}(s)}}{1-q^{a_{\Lambda^\circledast}(s)}t^{l_{\Lambda^*}(s)+1}}
\end{align}
where $\mathcal{B}\Lambda$ stands for the set of boxes in $\Lambda$ that do not lie at the same time in a row containing a circle and in a column containing a
circle, and where $a_{\lambda}(s)$ and $l_{\lambda}(s)$ correspond respectively to the arm-length and the leg-length of cell $s$ in the partition $\lambda$ (note that $\Lambda^*$ and $\Lambda^\circledast$ are partitions). In the special case with no fermionic variables, the extra $q$ power disappears and 
\eqref{ns} is exactly as in the usual Macdonald polynomial case \cite{Mac}.

To define an evaluation in superspace is not completely obvious given the presence of anticommuting variables. 
For $F(x;\theta)$ a symmetric function in superspace in $N$ variables and of degree $m$ in the anticommuting variables $\theta$, the evaluation 
${\mathcal E}_{t^{N-m},q,t}^m(F)$ is obtained in three steps:
\begin{enumerate}
\item Take the coefficient of $\theta_1 \cdots \theta_m$ in $F(x;\theta)$
\item Divide by the Vandermonde determinant $\prod_{1 \leq i \leq j \leq m} (x_i-x_j)$
\item Specialize the variables as
  $$x_1=\frac{1}{q^{m-1}}, x_2=\frac{t}{q^{m-2}},\dots, x_{m-1}=\frac{t^{m-2}}{q},x_m=t^{m-1},
  x_{m+1}=t^m,\dots, x_N=t^{N-1} $$
\end{enumerate}  
Up to powers of $q$ and $t$ (given explicitly in Theorem~\ref{theomain}), we have that
\begin{equation}
{\mathcal E}_{t^{N-m},q,t}^m \bigl(P_\La\bigr)= q^* t^* \, 
   \frac{\prod_{(i,j)\in \mathcal{S}\Lambda}(1-q^{j-1}t^{N-(i-1)})}{\prod_{s\in \mathcal{B}\Lambda} (1-q^{a_{\Lambda^\circledast}(s)}t^{l_{\Lambda^*}(s)+1})}
\end{equation}
where $\mathcal{S}\Lambda$ stands for the set of cells in the skew diagram $\Lambda^\cd/\delta_{m+1}$, with $\delta_{m+1}$ the staircase partition $(m,m-1,\dots,1,0)$.
This is our first evaluation formula.  In order to prove this formula recursively, we will
need another evaluation formula $\widetilde {\mathcal E}_{t^{N-m},q,t}^m \bigl(F(x;\theta) \bigr):= {\mathcal E}_{t^{N-m},q,t}^{m-1} \bigl(\partial_{\theta_n} F(x,\theta)|_{x_n=0} \bigr)$.  This second evaluation formula is such that
if $\Lambda$ is a superpartition with non zero fermionic degree, then
(up to powers of $q$ and $t$ given explicitly in Theorem~\ref{theomain2})
\begin{equation}
\widetilde{\mathcal E}_{t^{N-m},q,t}^m \bigl(P_\La\bigr)= q^* t^* \, 
   \frac{\prod_{(i,j)\in \mathcal{S}\tilde{\mathcal C}\Lambda}(1-q^{j-1}t^{N-1-(i-1)})}{\prod_{s\in \mathcal{B}\Lambda} (1-q^{a_{\Lambda^\circledast}(s)}t^{l_{\Lambda^*}(s)+1})}
\end{equation}
where $\tilde {\mathcal C} \Lambda$ is the superpartition obtained by removing
the circle in the first column of the diagram associated to $\Lambda$.

In order to prove the two formulas for the evaluation,
we have to replace $t^{N-m}$ in both evaluations by a formal parameter $u$ (which as seen in Section~\ref{sec5} necessitates a somewhat non-trivial evaluation).
Following the methods of \cite{lotharingian}, together with the recursions suggested in \cite{DLMeva}, we can then prove the two evaluation formulas and get, essentially as a corollary, the norm squared \eqref{ns}. 
This approach relies on first establishing
two fundamental recursions satisfied by Macdonald polynomials in superspace (see Proposition~\ref{OpColumnaRem} and \ref{OpColumnaRem2}) as well as showing that the terms appearing in the Pieri rules for the Macdonald polynomials in superspace are vertical strips (since this follows from properties of interpolation Macdonald polynomials, the proofs are relegated to Appendix~\ref{appenB}).

\section{Preliminaries}\label{sps}

\subsection{Symmetric polynomials in superspace} \cite{BDLM2,DLMeva}
A polynomial {in} superspace, or equivalently, a superpolynomial, is
a polynomial in the usual $N$ variables $x_1,\ldots ,x_N$  and the $N$ anticommuting variables $\ta_1,\ldots,\ta_N$ over a certain field, which will be taken throughout this article to be $\mathbb Q(q,t)$.  
A superpolynomial $P(x,\theta)$,
with $x=(x_1,\ldots,x_N)$ and $\theta=(\theta_1,\ldots,\theta_N)$, is said to be symmetric if the following is satisfied:
\begin{equation}
\mathcal{K}_{\sigma}P(x,\theta)=P(x,\theta) \qquad {\rm for~all~} \sigma \in S_N 
\, ,
\end{equation}
where
\begin{equation}
\mathcal{K}_{\sigma}=\kappa_{\sigma}K_{\sigma},
\qquad\text{with}\quad\begin{cases}
&K_{\sigma}\,:\, (x_1,\dots,x_N) \mapsto (x_{\sigma(1)}, \dots,
x_{\sigma(N)})
\\
&\kappa_{\sigma}\,\;:\, (\theta_1,\dots,\theta_N) \mapsto (\theta_{\sigma(1)}, \dots,
\theta_{\sigma(N)}).
\end{cases}\end{equation}
The space of symmetric superpolynomials in $N$ variables
over the field $\mathbb Q(q,t)$ will
be denoted $\mathscr R_N$, and its inverse limit by $\mathscr R$ (loosely speaking, the number of variables is considered infinite in $\mathscr R$).

Before defining superpartitions, we  recall some definitions
related to partitions \cite{Mac}.
A partition $\lambda=(\lambda_1,\lambda_2,\dots)$ of degree $|\lambda|$
is a vector of non-negative integers such that
$\lambda_i \geq \lambda_{i+1}$ for $i=1,2,\dots$ and such that
$\sum_i \lambda_i=|\lambda|$.  The length $\ell(\lambda)$
of $\lambda$ is the number of non-zero entries of $\lambda$.
Each partition $\lambda$ has an associated Ferrers diagram
with $\lambda_i$ lattice squares in the $i^{th}$ row,
from the top to bottom. Any lattice square in the Ferrers diagram
is called a cell (or simply a square), where the cell $(i,j)$ is in the $i$th row and $j$th
column of the diagram.  
The conjugate $\lambda'$ of  a partition $\lambda$ is represented  by
the diagram
obtained by reflecting  $\lambda$ about the main diagonal.
We say that the diagram $\mu$ is contained in $\la$, denoted
$\mu\subseteq \la$, if $\mu_i\leq \la_i$ for all $i$.  Finally,
$\la/\mu$ is a horizontal (resp. vertical) $n$-strip if $\mu \subseteq \lambda$, $|\lambda|-|\mu|=n$,
and the skew diagram $\la/\mu$ does not have two cells in the same column
(resp. row). 

Symmetric superpolynomials  are naturally indexed by superpartitions. A superpartition $\Lambda$ of
degree $(n|m)$ and length $\ell$
  is a pair $(\Lambda^\circledast,\Lambda^*)$ of partitions
$\Lambda^\circledast$ and $\Lambda^*$ such
 that:
 \begin{enumerate} \item $\Lambda^* \subseteq \Lambda^\circledast$;
 \item the degree of $\Lambda^*$ is $n$;
 \item the length of $\Lambda^\circledast$ is $\ell$;
 \item the skew diagram $\Lambda^\circledast/\Lambda^*$
is both a horizontal and a vertical $m$-strip.\footnote{Some authors call such a diagram an $m$-rook strip.}
 \end{enumerate}
We refer to  $m$ and $n$ respectively as the fermionic degree 
and total degree 
of $\La$.
 Obviously, if
$\Lambda^\circledast= \Lambda^*=\lambda$,
then $\Lambda=(\lambda,\lambda)$ can be interpreted as the partition
$\lambda$.

We will also need another characterization of a superpartition.
 A superpartition $\La$ is 
a pair of partitions $(\La^a; \La^s)=(\La_{1},\ldots,\La_m;\La_{m+1},\ldots,\La_\ell)$, 
 where $\La^a$ is a partition with $m$ 
distinct parts (one of them possibly  equal to zero),
and $\La^s$ is an ordinary partition.   The correspondence 
between $(\Lambda^\circledast,\Lambda^*)$ and 
$(\Lambda^a; \Lambda^s)$ is given explicitly as follows:
given 
$(\Lambda^\circledast,\Lambda^*)$, the parts of $\Lambda^a$ correspond to the
parts of $\Lambda^*$ such that $\Lambda^{\circledast}_i\neq 
\Lambda^*_i$, while the parts of $\Lambda^s$ correspond to the
parts of $\Lambda^*$ such that $\Lambda^{\circledast}_i=\Lambda^*_i$.

The conjugate of a superpartition 
$\Lambda=(\Lambda^\circledast,\Lambda^*)$ is $\Lambda'=((\Lambda^\circledast)',(\Lambda^*)')$.
A diagrammatic representation of $\La$ is given by 
the Ferrers diagram of $\La^*$ with circles added in the cells corresponding
to $\Lambda^{\circledast}/\Lambda^*$.
For instance, if $\La=(\Lambda^a;\Lambda^s)
=(3,1,0;2,1)$,  we have $\Lambda^\circledast=(4,2,2,1,1)$ and $\Lambda^*
=(3,2,1,1)$, so that 
{\small
\begin{equation*} \label{exdia}
     \La^\cd:\quad{\tableau[scY]{&&&\\&\\&\\\\ \\ }} \quad
         \La^*:\quad{\tableau[scY]{&&\\&\\ \\ \\ }} \quad
 \Longrightarrow\quad      \La:\quad {\tableau[scY]{&&&\bl\tcercle{}\\&\\&\bl\tcercle{}\\ \\
    \bl\tcercle{}}} \quad    \La':\quad {\tableau[scY]{&&&&\bl\tcercle{}\\&&\bl\tcercle{}\\ \\
    \bl\tcercle{}}},
\end{equation*}}
\hspace{-0.3cm} where the last diagram illustrates the conjugation operation that corresponds, as usual, to replacing rows by columns.

The extension of the dominance ordering
to superpartitions  is \cite{DLMeva}:
\begin{equation} \label{eqorder1}
 \Omega\leq\Lambda \quad \text{iff}\quad
 \deg(\La)=\deg(\Om) ,
 \quad \Omega^* \leq \Lambda^*\quad \text{and}\quad
\Omega^{\circledast} \leq  \Lambda^{\circledast}.
\end{equation}
Note that comparing two superpartitions amounts to comparing  two pairs of ordinary partitions,
($\Omega^*$, $\La^*$) and
($\Omega^{\circledast} $, $ \La^{\circledast}$), 
with respect to the usual  dominance ordering: 
\begin{equation}\label{ordre}
   \mu \leq \la\quad\text{ iff }\quad |\mu|=|\la|\quad\text{
and }\quad \mu_1 + \cdots + \mu_i \leq \lambda_1 + \cdots + \lambda_i\quad \forall i \, . \end{equation}

Two simple bases  of the space of 
symmetric polynomials in superspace (with commuting indeterminates $x_1,\ldots ,x_N$ and anticommuting  indeterminates $\ta_1,\ldots,\ta_N$) will be particularly relevant to our work:\begin{enumerate}\item 
the extension of the monomial symmetric 
functions, $m_\La=m_\La(x,\theta)$, defined by
\begin{equation}
m_\La={\sum_{\sigma \in S_N} }' \mathcal{K}_\sigma \left(\theta_1\cdots\theta_m x_1^{\La_1}\cdots x_\ell^{\La_\ell}\right),
\end{equation}
where the sum is over the  permutations of $\{1,\ldots,N\}$ that produce distinct terms;
\item 
 the generalization of the elementary
symmetric
functions, $e_\La=e_\La(x,\theta)$, defined by
\begin{equation}\label{basee}
e_\La=\tilde{e}_{\La_1}\cdots\tilde{e}_{\La_m}e_{\La_{m+1}}\cdots e_{\La_{\ell}}\, ,\end{equation}
{where}\begin{equation} \tilde{e}_k=m_{(0;1^k)}\qquad\text{and}\qquad e_r=
m_{(\emptyset;1^r)} \, ,
\end{equation}
with $k\geq 0$ and $r \geq 1$.
\item 
 the generalization of the power-sum 
symmetric
functions, $p_\La=p_\La(x,\theta)$, defined by
\begin{equation}\label{spower}
p_\La=\tilde{p}_{\La_1}\cdots\tilde{p}_{\La_m}p_{\La_{m+1}}\cdots p_{\La_{\ell}}\, ,\end{equation}
{where}\begin{equation} \tilde{p}_k=\sum_{i=1}^N\theta_ix_i^k\qquad\text{and}\qquad p_r=
\sum_{i=1}^Nx_i^r \, ,
\end{equation}  
with $k\geq 0$ and $r \geq 1$.
\end{enumerate}


\subsection{The non-symmetric Macdonald polynomials}\label{nonS}

The ordinary Macdonald polynomials can be defined by the conditions (1) and (2) in \eqref{mac1}.
But they could alternatively be defined directly in terms of the so-called non-symmetric Macdonald    polynomials by a suitable symmetrization process \cite{Mac1,Che} (see also \cite{Mac2,Mar}). As will be shown in the following section, this can also be done for their superspace extension.  But since this result uses a fair amount of notations and definitions, it is convenient to summarize these here. 

The non-symmetric Macdonald polynomials are defined in terms of an eigenvalue problem formulated in terms of the Cherednik operators \cite{Che}. They are constructed from the 
operators $T_i$ 
defined as
\begin{equation}T_i=t+\frac{tx_i-x_{i+1}}{x_i-x_{i+1}}(K_{i,i+1}-1),\quad i=1,\ldots,N-1,\end{equation}
and
\begin{equation}
 {T_0=t+\frac{qtx_N-x_1}{qx_N-x_1}(K_{1,N}\tau_1\tau_N^{-1}-1)}\, ,
\end{equation}
where we recall that $ {K_{i,j}}$ exchanges the variables $x_i$ and ${x_{j}}$.
Note that for $t=1$, $T_i$ reduces to $K_{i,i+1}$. 
The $T_i$'s satisfy the  {affine} Hecke algebra relations  {($0\leq i\leq N-1$)}:
\begin{align}\label{Hekoalg}&(T_i-t)(T_i+1)=0\nonumber\\
&T_iT_{i+1}T_i=T_{i+1}T_iT_{i+1}\nonumber\\
&T_iT_j=T_jT_i \, ,\quad i-j \neq \pm 1 \mod N
\end{align}
where the indices are taken modulo $N$.
To define the Cherednik operators, we also need to introduce the
  $q$-shift operators 
  \begin{equation} \tau_i:\begin{cases}x_i\mapsto qx_i,\\ x_j\mapsto x_j\;
{\rm~if~} j\neq i, \end{cases}
\end{equation}
and the operator $\omega$ defined as:  
\begin{equation}\omega=K_{N-1,N}\cdots K_{1,2} \, \tau_1.
\end{equation}
We note that $\omega T_i=T_{i-1}\omega$ for $i=2,\dots,N-1$.

We are now in position to define the Cherednik operators:
\begin{equation}Y_i=t^{-N+i}T_i\cdots T_{N-1}\omega T_1^{-1}\cdots T_{i-1}^{-1},\end{equation}
where 
$ T_j^{-1}$ (also denoted  $\bar T_j$ below) is
\begin{equation}\label{Tinv}
 T_j^{-1}=t^{-1}-1+t^{-1}T_j,
 \end{equation} 
which follows from the quadratic relation \eqref{Hekoalg} of the  Hecke algebra.
 These  operators satisfy the following  relations  \cite{Che, Kiri} 
: \begin{align} \label{tsym1}
T_i \, Y_i&= Y_{i+1}T_i+(t-1)Y_i\nonumber \\
T_i \, Y_{i+1}&= Y_{i}T_i-(t-1)Y_i\nonumber \\
T_i Y_j & = Y_j T_i \quad {\rm if~} j\neq i,i+1.
\end{align}
It can be easily deduced from these relations that
\begin{equation}\label{TYi}
(Y_i+Y_{i+1})T_i= T_i (Y_i+Y_{i+1}) \qquad {\rm and } \qquad (Y_i Y_{i+1}) T_i =
T_i (Y_i Y_{i+1}). 
\end{equation}
But more importantly, the $Y_i$'s commute among each others, $[Y_i,Y_j]=0$,
and can therefore be simultaneously diagonalized. Their eigenfunctions are the
 (monic) non-symmetric Macdonald polynomials (labeled by compositions).
To be more precise,  the non-symmetric Macdonald polynomial $E_\eta$ is 
the 
unique polynomial with rational coefficients in $q$ and $t$ 
that is triangularly related to the monomials (in the Bruhat ordering on compositions)
\begin{align}\label{defEtrian}
E_\eta=x^\eta+\sum_{\nu\prec\eta}b_{\eta\nu}x^\nu
\end{align}
and that satisfies, for all $i=1,\dots,N$, 
\begin{align}Y_i E_\eta=\bar \eta_iE_\eta,\qquad\text{where}\qquad  \bar\eta_i =q^{\eta_i}t^{-\bar l_\eta(i)} \label{eigenvalY}
\end{align}
with $\bar l_\eta(i)=\# \{k<i|\eta_k\geq \eta_i\}+\# \{k>i|\eta_k> \eta_i\}$.
 The Bruhat order on compositions is defined as follows:
 \begin{align}\nu\prec\eta\quad \text{ {iff}}\quad \nu^+<\eta^+\quad \text{or} \quad \nu^+=\eta^+\quad \text{and}\quad w_\eta < w_\nu,
 \end{align}
 where $\eta^+$ is the partition associated to $\eta$ and 
$w_{\eta}$ is the unique permutation of minimal length such 
that $\eta = w_{\eta} \eta^+$ ($w_{\eta}$ permutes the entries of $\eta^+$).
In the Bruhat order on the symmetric group, $w_\eta {<} w_\nu$ iff
$w_{\eta}$ can be obtained as a  {proper} subword  of $w_{\nu}$.

The following two properties of the non-symmetric Macdonald polynomials will be needed below. 
The first one expresses the stability of the polynomials $E_\eta$ with respect to the number of variables  (see e.g. \cite[eq. (3.2)]{Mar}):
\begin{equation} \label{property1}
E_\eta (x_1,\dots,x_{N-1},0) =
\left \{ 
\begin{array}{ll}
E_{\eta_-} (x_1,\dots,x_{N-1})
& {\rm if~} 
\eta_N = 0\, , \\
0 & {\rm if~} \eta_N \neq 0\, .
\end{array} \right.
 \end{equation}
where $\eta_{-}=(\eta_1,\ldots, \eta_{N-1}) $.   
The second one gives the action of the operators $T_i$ on $E_\eta$ (see e.g.  \cite{BF}): 
\begin{equation} \label{property2}
T_i E_{\eta} = \left\{ 
\begin{array}{ll}
\left(\frac{t-1}{1-\delta_{i,\eta}^{-1}}\right) E_\eta + t E_{s_i \eta} & {\rm if~} 
\eta_i < \eta_{i+1} \, ,  \\
t E_{\eta} &  {\rm if~} 
\eta_i = \eta_{i+1} \, ,\\
\left(\frac{t-1}{1-\delta_{i,\eta}^{-1}}\right) E_\eta + \frac{(1-t{\delta_{i,\eta}})(1-t^{-1}\delta_{i,\eta})}{(1-{\delta_{i,\eta}})^2} E_{s_i \eta} & {\rm if~} 
\eta_i > \eta_{i+1} \, ,
\end{array} \right. 
\end{equation}
where $\delta_{i,\eta}=\bar \eta_i/\bar \eta_{i+1}$ {and} $s_i \eta=(\eta_1,\dots,\eta_{i-1},\eta_{i+1},\eta_i,\eta_{i+2},\dots,\eta_N)$.

Finally, we introduce the $t$-symmetrization and $t$-antisymmetrization operators  \cite{Mac1}:
\begin{equation}\label{upm}
 U^+_N=\sum_{\sigma\in S_N}T_\sigma\qquad\text{and}\qquad  U^-_N=\sum_{\sigma\in S_N}(-t)^{-\ell(\sigma)}T_\sigma \end{equation}
where
\begin{equation}T_\sigma=T_{i_1}\cdots T_{i_\ell}
  \qquad  \text{if}\quad \sigma=s_{i_1}\cdots s_{i_\ell}.\end{equation}
The $t$-symmetrization and $t$-antisymmetrization operators obey the relations
\begin{align}  \label{Urel} 
T_i  U^+_N =  U^+_N T_i = t  \, U^+_N \qquad {\rm and} \qquad T_i  U^-_N =  U^-_N T_i =  - U^-_N
\end{align}  
Note that for any polynomial $f$ in the variables $x_1,\ldots,x_N$, we have 
$K_{i,i+1} U^+_Nf=U^+_Nf$,  but $ K_{i,i+1} U^-_Nf\neq -U_N^-f$ since \cite[eq.(2.26)]{Mar}
\begin{equation}\label{UvsA}
U^-_N\,f=t^{-\binom{N}{2}}\frac{\Delta^t_{N}}{\Delta_{N}}A_N\,{f},
\end{equation} 
where  
\begin{equation}\label{defdel} 
A_N=\sum_{\sigma \in S_N}(-1)^{\ell(\sigma)}K_\sigma\, , \quad  \Delta^t_N=\prod_{1\leq i<j\leq N}(tx_i-x_j)\, ,\quad  
\Delta_{N}=\Delta^1_N\, . 
\end{equation}
Note that $A_N$ is the usual antisymmetrization operator.
Below, we will designate by $S_m$ and $S_{m^c}$ the group of permutations of the variables $x_1,\ldots,x_m$ and $x_{m+1},\ldots,x_N$ respectively. For instance, $U^-_{m}$ and $U^+_{m^c}$ are defined as in \eqref{upm} but with $S_N$ replaced by $S_{m}$ and $S_{m^c}$ respectively. Similarly, we will frequently use the notation $\Delta^t_m$ which is defined as in \eqref{defdel} but with $N$ replaced by $m$.


\subsection{Macdonald superpolynomials} \label{PvsEs}
All the result of this section can be found in \cite{BDLM2}.
We first define the  Macdonald superpolynomials
in terms of the non-symmetric Macdonald polynomials (it was proven in \cite{BDLM2} that they correspond to those of Theorem~\ref{theo1}).

\begin{definition}  The Macdonald superpolynomials $P_{\Lambda}={  P}_\La(x,\theta;q,t)$ are such that 
 \begin{equation}\label{PvsE} 
 {  P}_\La=
\frac{(-1)^{\binom{m}{2}}}{f_{\La^s}(t)\, t^{{\rm inv} (\La^s)} }
\sum_{\sigma \in S_N/(S_m\times S_{m^c})}\mathcal{K}_\sigma \theta_1\cdots\theta_m A_{m}U^+_{m^c}
E_{\Lambda^R} \, ,\end{equation}
where 
\begin{equation}\label{cla}
f_{\La^s}(t)=\prod_{{j\geq0}} [n_{\La^s}(j)]_t ! \, ,
\end{equation} 
with $n_{\La^s}(j)$ being the number of occurrences of $j$ in $\La^s$ and $\Lambda^R$ stands for the concatenation of $\La^a$ and $\La^s$ read in reverse order: \begin{equation}\Lambda^R=(\La_m,\ldots,\La_1,\La_N,\ldots,\La_{m+1})\, .\end{equation}
\end{definition}

 {In \eqref{PvsE}, we extended} the usual concept of inversion on a permutation to
a partition:  ${\inv}(\La^s)$ is the number of inversions in $\La^s$, the latter number being equal to
\begin{equation} \inv(\la)=\# \{n\geq i>j\, |\, \la_i<\la_j\} \, , \end{equation}
where $n$ is the number of entries in $\lambda$ (including 0's).
For instance, we have $\inv(22100)=8$.  In \eqref{cla}, we also used the following standard notation:
$$[k]_t!=[1]_t[2]_t \cdots [k]_t  \qquad {\rm with} \qquad
[m]_t=(1-t^m)/(1-t) \, .
$$ 


We first show that the stability of $E_\eta$ with respect to the number of variables can be lifted to that of $  P_{\Lambda}$.

 \begin{proposition}\label{propostable}
Suppose that $N>m$.  Then
the Macdonald superpolynomials $  P_{\Lambda}$ are stable with respect the number of variables, that is,
\begin{multline}
{  P}_{\Lambda}(x_1,\dots,x_{N-1},0,\theta_1,\dots,\theta_{N-1},0) =
\left \{ 
\begin{array}{ll}
{  P}_{\Lambda_-}(x_1,\dots,x_{N-1},\theta_1,\dots,\theta_{N-1}) & {\rm if~} 
\Lambda_N = 0\, , \\
0 & {\rm if~} \Lambda_N \neq 0\, ,
\end{array} \right.
\end{multline}
where $\Lambda_-=(\Lambda_1,\dots, \La_m ; \La_{m+1}, \ldots, \Lambda_{N-1})$.
\end{proposition}

We now provide a characterization of the $P_{\Lambda}$'s as common 
eigenfunctions
of two commuting operators:
\begin{align}
D_1^*&=\sum_{m=0}^N
\sum_{\sigma \in S_N/(S_m\times S_{m^c})}\mathcal{K}_\sigma\left(\frac{\Delta_m}{\Delta^t_m}\left(Y_1 + \cdots +Y_N \right)
\frac{\Delta^t_{m}}{\Delta_{m}}\pi_{1,\ldots,m}\right), \\
D_1^\cd&=\sum_{m=0}^N
\sum_{\sigma \in S_N/(S_m\times S_{m^c})}\mathcal{K}_\sigma\left(\frac{\Delta_{m}}{\Delta^t_{m}}\left(
qY_1 + \cdots +qY_m + Y_{m+1} + \cdots +Y_N \right)
\frac{\Delta^t_{m}}{\Delta_{m}}\pi_{1,\ldots,m}\right). 
\end{align}
where the operator $\pi_{1,\dots,m}$ is the  projection operator defined as
\begin{equation}\pi_{1,\dots,m}=\prod_{i= 1}^m\theta_i\partial_{\theta_i}\prod_{j=m+1}^N \partial_{\theta_j}{\theta_j}\, . \end{equation} 
In this equation,  $\partial_{\theta_i}$ denotes  the standard  derivative with respect to the Grassmann variable $\theta_i$, which is a linear operator such that,  for all polynomials $f=f(x,\theta)$ and $i,j\in\{1,\ldots, N\}$, 
\begin{equation} \partial_{\theta_i}\left(x_j f\right) = x_j\partial_{\theta_i}\left(f\right)\,\qquad   \partial_{\theta_i}\left(\theta_j f\right)=\delta_{i,j}\,f-\theta_j\,\partial_{\theta_i}\left(f\right)\, ,\end{equation}
and \beq \partial_{\theta_i}\partial_{\theta_j}\left(f\right)=-\partial_{\theta_j}\partial_{\theta_i}\left(f\right)\quad \Longrightarrow\quad \partial_{\theta_i}^2(f)=0\,.\eeq 
It is easy to see that
 \begin{equation}\pi_{1,\dots,m} \ta_{i_1}\cdots\ta_{i_k}=\begin{cases}\theta_1 \cdots \theta_m 
&{\rm if~}\{i_1,\dots,i_k \}=\{1,\dots,m \}\, ,\\0&{\rm if~}\{i_1,\dots,i_k \}\neq \{1,\dots,m \}.
 \end{cases}\end{equation}

 The eigenvalues of the operators $D_1^*$ and $D_1^\circledast$ when acting on $P_\La$ are 
 \begin{equation}\label{eigD}
   D_1^* P_\La=\varepsilon_{\La^*}P_\La \qquad \text{and} \qquad
D_1^\circledast P_\La=\varepsilon_{\La^\circledast}P_\La,
 \end{equation}
 where $\varepsilon_{\la}=\sum_{i=1}^N q^{\la_i}t^{1-i}$.  Since the two eigenvalues completely determine $\Lambda$, the following lemma holds.
\begin{lemma} \label{lemmachar}
The Macdonald polynomial in superspace $P_\Lambda$ can be characterized as
 the unique common eigenfunction of $D_1^*$ and $D_1^\circledast$ with  eigenvalues
 $\varepsilon_{\La^*}$ and $\varepsilon_{\La^\circledast}$ respectively.
\end{lemma}
 
 We now state the generalization to superspace of  the
 standard duality property that relates the  Macdonald symmetric functions $P_\la(q,t)$ and $P_{\la'}(t,q)$ \cite[Section VI.5]{Mac}.  Let $\Omega_{q,t}$
 be the homomorphism defined as
 \begin{equation} \label{defomqt}
\Omega_{q,t} p_r=(-1)^{r-1}\frac{1-q^r}{1-t^r}\, p_r\qquad \Omega_{q,t}\tilde p_r=(-q)^{r}\tilde p_r \, ,
 \end{equation}
which is such that
\beq \label{actionomegaqt}
\Omega_{q,t}\,p_\La= (-1)^{|\Lambda|-\ell(\La^s)}  q^{|\La^a|}\prod_{i=1}^{\ell(\La^s)}
\frac{1-q^{\La_i^s}}{1-t^{\La^s_i}}\,p_\La,\eeq
\begin{theorem}\label{theoduality}  Let $Q_\Lambda=P_\Lambda/\|P_\La\|^2$, where
  $\|P_\La\|^2= \LL  P_{\La}| P_{\La} \RR_{q,t}$.
  Then, the following duality holds\footnote{The corresponding formula in  \cite{BDLM2}  does not have the sign $(-1)^{\binom{m}{2}}$.  This is  due to the fact that our choice of scalar product  differs from theirs by $(-1)^{\binom{m}{2}}$.}: 
\beq
\Omega_{q,t} P_\La(q,t)= (-1)^{\binom{m}{2}}(qt^{-1})^{|\La|} \, Q_{\La'}(t^{-1},q^{-1}).\eeq
\end{theorem}

\section{Skew  Macdonald polynomials in superspace}

We define the coefficients $g_{\Om,\Gamma}^\La$ by
\begin{equation}\label{EqProdJ}
P_\Om \, P_\Gamma= \sum_\La \frac{1}{{\|P_\La\|^2}} \,g_{\Om,\Gamma}^\La\, P_\La.
\end{equation}
By orthogonality, this is equivalent to saying that
\begin{equation}
g_{\Om,\Gamma}^\La =\LL {P_\La}   |  {P_\Om} {P_\Gamma} \RR.\end{equation}

The skew Macdonald polynomial $P_{\La/\Om}$ is now defined as the unique symmetric superfunction in $x$ and $\theta$  such that
\begin{equation}\label{defskew}
g^{\La}_{\Om \Gamma}=\LL {P}_{\La/\Om}\,  | \,{ P}_{\Gamma}\RR=\LL {P}_\La   |  {P}_\Om {P}_\Gamma \RR  \,.
\end{equation}
Observe that this definition is equivalent to
\begin{equation}
P_{\Lambda/\Omega} = \sum_{\Gamma} \frac{g^{\Lambda}_{\Omega \Gamma}}{\|P_\Gamma\|^2}
P_{\Gamma} \, .
\end{equation}

The following proposition is proved exactly as in the case of the
Jack polynomials in superspace. 
\begin{proposition}\label{PropSkew}
 Let $(x,y;\theta,\phi)$ denote the ordered set
$(x_1,x_2,\ldots,y_1,y_2,\ldots;\theta_1,\theta_2,\ldots,\phi_1,\phi_2,\ldots).$
Then,
we have
\begin{equation}\label{EqSkew}
{P}_\Gamma(x,y;\theta,\phi)=\sum_\La\frac{1}{\|P_\La\|^2}\,{P}_\La(x;\theta)
{P}_{\Gamma/\La}(y;\phi)=\sum_\La\frac{1}{\|P_\La\|^2}\,{P}_{\Gamma/\La}(x;\theta)
{P}_{\La}(y;\phi) .
\end{equation}
\end{proposition}

The following lemma is an immediate consequence of the duality induced by $\Omega_{q,t}$ stated in Theorem~\ref{theoduality}.
\begin{lemma} \label{lemmaequivconj}
We have that
\begin{equation}\label{lessi}
g^{\La}_{\Om \Gamma}\neq 0\quad\text{ if and only if }
\quad g^{\La'}_{\Om' \Gamma'}\neq 0.
\end{equation}
\end{lemma}

\subsection{Necessary conditions for the non-vanishing of coefficients
in the Pieri rule: horizontal and vertical strips}\label{sectstrips}
Let $n$ and $\tilde n$ refer respectively to the superpartitions
$(n)$ and $(n;)$, i.e.,   associated respectively to the following diagrams both containing $n$  squares:
\begin{equation}  n={\tableau[scY]{&&\bl\;\cdots&\bl&}}\quad\text{and}\quad
\tilde n={\tableau[scY]{&&\bl\;\cdots&\bl& &\bl\tcercle{}}}.
\end{equation}
We now obtain  necessary conditions for the
non-vanishing of the coefficients $g^\La_{\Omega,n}$ and
$g^\La_{\Omega,\tilde n}$.
These results specify -- without evaluating them explicitly -- the
coefficients that can appear in a Pieri-type rule for Jack polynomials
in superspace.

When no fermions are involved (in which case superpartitions
$\Lambda$ and $\Om$ are usual partitions $\la$ and $\mu$),
it is known that the coefficient $g^\la_{\mu,n}\neq 0$ if and only if
$\la/\mu$ is a horizontal $n$-strip.
The concept of horizontal or vertical strip can be easily generalized to superpartitions.

\begin{definition} \label{strips}
We say that $\La/\Om$
is a horizontal $n$-strip
if
$\La^*/\Om^*$  and
$\La^{\circledast}/\Om^{\circledast}$
are both horizontal $n$-strips.
Similarly, we say that $\La/\Om$
is a horizontal  $\tilde n$-strip
if
$\La^*/\Om^*$ is a horizontal $n$-strip and
$\La^{\circledast}/\Om^{\circledast}$
is a horizontal  $n+1$-strip. The definitions are similar for vertical strips.
\end{definition}
Consider for example, $\La=(4,1;2,1)$ and $\Om=(2,0;3,1)$.  Then, as
 illustrated in Figure \ref{hstrip}, $\La/\Om$ is a horizontal 3-strip, but it is not a vertical 3-strip.  Similarly, it is readily {seen}
from Figure \ref{vstrip} that $(3,0;2,1)/(2;2)$ is  a vertical $\tilde{2}$-strip.

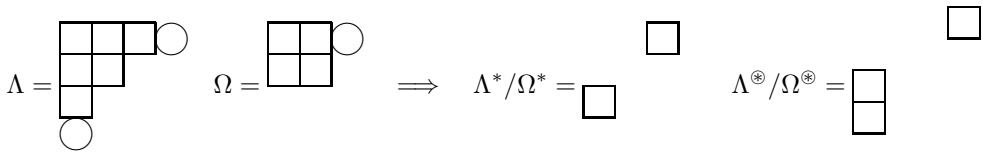
\begin{figure}[h]\caption{Horizontal $n$-strip}\label{hstrip}
{\begin{equation*}
\La={\tableau[scY]{&&&&\bl\tcercle{}\\&\\&\bl\tcercle{}\\&\bl}}\quad \Om={\tableau[scY]{&&\\&\bl\tcercle{}\\&\bl\\\bl\tcercle{}}} \quad\Longrightarrow\quad
 \La^*/\Om^*={\tableau[scY]{\bl&\bl&\bl&\\  \bl& \\ \bl & \bl  \\ &\bl  }}
\qquad
\La^\circledast/\Om^\circledast={\tableau[scY]{  \bl&\bl&\bl&&\\  \bl& \bl\\  \bl & \\ \bl }}
\end{equation*}}
\end{figure}

\begin{figure}[h]
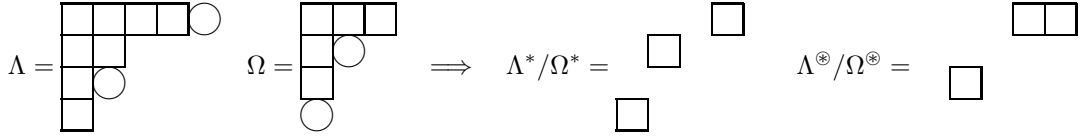
\caption{Vertical $\tilde n$-strip}\label{vstrip}
{\begin{equation*}
\La={\tableau[scY]{&&&\bl\tcercle{}\\&\\ \\  \bl\tcercle{} }}\quad \Om={\tableau[scY]{&&\bl\tcercle{}\\&\\  \bl \\\bl }} \quad\Longrightarrow\quad
\La^*/\Om^*={\tableau[scY]{\bl&\bl&\\  \bl& \bl \\ &\bl \\ \bl  }}
\qquad
\La^\circledast/\Om^\circledast={\tableau[scY]{\bl&\bl&\bl&\\  \bl& \bl\\   \\  \\ \bl}}
\end{equation*}}
\end{figure}

The proofs of the next two propositions 
are rather involved. As such, they are
relegated to Appendix~\ref{appenB}. Note
that the equivalences in the statements follow from
Lemma~\ref{lemmaequivconj}.
\begin{proposition}\label{PropgI} 
The coefficient $g^\La_{\Om,n}\neq0$
 only if $\La/\Om$ is a horizontal $n$-strip.  Equivalently,
the coefficient $g^\La_{\Om,1^n}\neq0$
 only if $\La/\Om$ is a vertical $n$-strip
\end{proposition}

\begin{proposition} \label{PropgII}
The coefficient  $g^\La_{\Om,\tilde n}\neq0$
only if
$\La/\Om$ is a horizontal $\tilde n$-strip. Equivalently,
the coefficient  $g^\La_{\Om,(0;1^n)}\neq0$
only if
$\La/\Om$ is a vertical $\tilde n$-strip
\end{proposition}

Recall that the diagram $\mu$ is contained in $\la$, denoted
$\mu\subseteq \la$, if $\mu_i\leq \la_i$ for all $i$.
For superpartitions
we define $\Om \subseteq \La$ as follows:
\begin{equation}\Om \subseteq\La\quad \text{if and only if} \quad\Om^* \subseteq \La^*\quad\text{and}\quad\Om^{\circledast} \subseteq \La^{\circledast} \, .
\end{equation}
{For instance, $(0;3,2)\subset(3,0;3,1)$ but $(2,1;3)\not\subseteq(3,0;3,1)$}.
Since $P_{(\emptyset;1^n)}=e_n$ and  $P_{(0;1^n)}=\tilde e_n$ and since
the $e_{\Lambda}$'s form a multiplicative basis, the previous propositions have the following corollary.
\begin{corollary} \label{coro11}
We have that $g_{\Om \Gamma}^{\La}$ is zero unless $\Om \subseteq \La$ and $\Gamma \subseteq \La$.
\end{corollary}

\section{Operations on the first column}

We define two operations on the first column of a superpartition $\Lambda$ of fermionic degree $m$.
If the first column in the diagram of $\Lambda$
does not contain a circle (that is, if $\Lambda_m\neq 0$), then we let ${\mathcal C} \Lambda$ be the superpartition whose diagram is that of $\Lambda$ without its first column.  

\begin{figure*}[h]
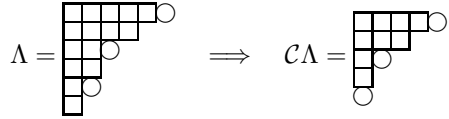
\caption{Operation $\mathcal{C}$}
{\begin{equation*}
\Lambda={\tiny{\tableau[scY]{&&&&&\bl\tcercle{}\\&&&\\&&\bl\tcercle{}\\ &&\bl \\&\bl\tcercle{}\\ &\bl}}}\quad\Longrightarrow\quad
 \mathcal{C}\Lambda={\tiny{\tableau[scY]{ &&& &\bl\tcercle{} \\&& \\ &\bl\tcercle{} \\&\bl \\ \bl \tcercle{} }}}
\end{equation*}}
\end{figure*}

If the first column in the diagram of $\Lambda$
 contains a circle  (that is, if $\Lambda_m= 0$), then we let $\tilde {\mathcal C} \Lambda$ be the superpartition whose diagram is that of $\Lambda$ without the circle in the first column.

\begin{figure*}[h]
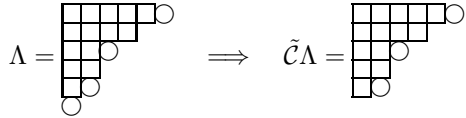
\caption{Operation $\tilde{\mathcal{C}}$}
{\begin{equation*}
\Lambda={\tiny{\tableau[scY]{&&&&&\bl\tcercle{}\\&&&\\&&\bl\tcercle{}\\ &&\bl \\&\bl\tcercle{}\\ \bl\tcercle{}}}}\quad\Longrightarrow\quad
\tilde{ \mathcal{C}}\Lambda={\tiny{\tableau[scY]{ &&&&&\bl\tcercle{}\\&&&\\&&\bl\tcercle{}\\ &&\bl \\&\bl\tcercle{}\\ \bl} }}
\end{equation*}}
\end{figure*}

In this manner, we have that if the first column in the diagram of $\Lambda$
 contains a circle, then $\mathcal C \tilde {\mathcal C} \Lambda$ is the superpartition whose diagram is that of $\Lambda$ without its first column. 

\begin{figure*}[h]
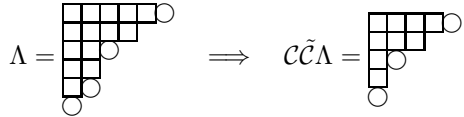
\caption{Operation $\mathcal{C}\tilde{\mathcal{C}}$}
{\begin{equation*}
\Lambda={\tiny{\tableau[scY]{&&&&&\bl\tcercle{}\\&&&\\&&\bl\tcercle{}\\ &&\bl \\&\bl\tcercle{}\\ \bl\tcercle{}}}}\quad\Longrightarrow\quad
\mathcal{C}\tilde{ \mathcal{C}}\Lambda={\tiny{\tableau[scY]{ &&&&\bl\tcercle{}\\&&\\&\bl\tcercle{}\\ &\bl \\\bl\tcercle{}\\  }}}
\end{equation*}}
\end{figure*}

\begin{proposition}\label{OpColumnaRem}
  Let $\Lambda$ be a superpartition whose diagram does not contain a circle in the first column. Then
\begin{align*}
P_{\Lambda}(x_1,\ldots,x_\ell;\theta_1,\ldots,\theta_\ell)=x_1 \cdots x_\ell P_{\mathcal C\Lambda}(x_1,\ldots,x_\ell;\theta_1,\ldots,\theta_\ell)
\end{align*}
where $\ell=\ell(\Lambda)=\ell(\Lambda^*)$.
\end{proposition}
\begin{proof}  From Lemma~\ref{lemmachar}, it suffices to show that $x_1 \cdots x_\ell P_{\mathcal C\Lambda}(x_1,\ldots,x_\ell;\theta_1,\ldots,\theta_\ell)$ is an eigenfunction of $D_1^*$ and $D_1^\circledast$ with eigenvalues $\varepsilon_{\Lambda^*}$
  and $\varepsilon_{\Lambda^\circledast}$ respectively.

  It is easy to show that
  \begin{equation}
    T_i  \, x_1 \cdots x_\ell=  x_1 \cdots x_\ell \, T_i  \qquad {\rm for~} i=1,\dots,\ell-1
\end{equation}
    (and similarly for $T_i^{-1}$).  Hence, in $N=\ell$ variables,
  we have
\begin{equation}
    Y_i  \, x_1 \cdots x_\ell=  q \, x_1 \cdots x_\ell \, Y_i  \qquad {\rm for~} i=1,\dots,\ell-1
\end{equation}
From the definition of $D_1^*$ and $D_1^\circledast$, it is then immediate that
in $N=\ell$ variables we also have
\begin{equation}
  D_1^*  \, x_1 \cdots x_\ell=  q \, x_1 \cdots x_\ell \, D_1^*  \qquad {\rm and}
  \qquad  D_1^\circledast \, x_1 \cdots x_\ell=  q \, x_1 \cdots x_\ell \, D_1^\circledast
\end{equation}
Therefore,
\begin{equation}
D_1^* \left( x_1 \cdots x_\ell P_{\mathcal C\Lambda}  \right)  = q\,  x_1 \cdots x_\ell D_1^* P_{\mathcal C\Lambda}= q \, \varepsilon_{\mathcal C \Lambda^*} \left( x_1 \cdots x_\ell P_{\mathcal C\Lambda}  \right)= \varepsilon_{\Lambda^*} \left( x_1 \cdots x_\ell P_{\mathcal C\Lambda}  \right)    
\end{equation}  
and similarly for $D_1^\circledast$.  We can thus conclude that $P_\Lambda=x_1 \cdots x_\ell P_{\mathcal C\Lambda}$. 
\end{proof}

The following lemma concerning non-symmetric Macdonald polynomials, which is needed to prove the next proposition, is certainly known.  But for lack of a proper reference, we provide a proof.
\begin{lemma} \label{lemNSZ}
Let $\eta=(\eta_1,\ldots,\eta_N)$ be a composition such that $\eta_i=0$ and $\eta_j\neq 0$ whenever $j\neq i$. Then 
\begin{equation}
  E_{\eta}(x_1,\ldots,x_N)\big|_{x_i=0}=E_{\eta_i^-} (x_1,\ldots,x_{i-1},x_{i+1},\ldots,x_N)
\end{equation}
and
\begin{equation}
E_{\eta}(x_1,\ldots,x_N)\big|_{x_j=0} =0 \quad {\rm if ~} j>i.
\end{equation}
where $\eta_i^-=(\eta_1,\ldots,\eta_{i-1}, \eta_{i+1},\ldots,\eta_N)$
is the composition $\eta$ without its $i$-th entry.
\end{lemma}

\begin{proof} We proceed by induction. We know from \eqref{property1}
that when $i=N$, we have
\begin{equation}
  E_{\eta}(x_1,\ldots,x_N)\big|_{x_N=0}=E_{\eta_N^-} (x_1,\ldots,x_{N-1})
\end{equation}
while when $i<N$ and $j=N$, we have
\begin{equation}
E_{\eta}(x_1,\ldots,x_N)\big|_{x_{N}=0}=0.
\end{equation}
Now,  suppose by induction that when $i=k$ we have
\begin{equation} \label{toprove1}
  E_{\eta}(x_1,\ldots,x_N)\big|_{x_{k}=0}=E_{\eta_k^-} (x_1,\ldots,x_{k-1},x_{k+1},\ldots,x_{N})
\end{equation}
while when $i<k$ and $j=k$, we have
\begin{equation} \label{toprove2}
E_{\eta}(x_1,\ldots,x_N)\big|_{x_{k}=0}=0.
\end{equation}
We thus  need to prove that \eqref{toprove1} and \eqref{toprove2}
still hold when $k$ is replaced by $k-1$.   We first consider the case \eqref{toprove1}. Suppose that $i=k-1$.  Since $\eta_{k-1}=0$ and $\eta_k \neq 0$ by hypothesis, we have from \eqref{property2} that   
\begin{equation} \label{eqrhs1}
  T_{k-1} E_\eta (x_1,\ldots,x_N) = c_\eta \, E_\eta (x_1,\ldots,x_N)+tE_{s_{k-1} \eta}(x_1,\ldots,x_N),
\end{equation}
where $c_\eta$ is an irrelevant constant.  By definition of $T_{k-1}$, we also have
\begin{align} \label{eqrhs2}
  T_{k-1} E_{\eta}(x_1,\ldots,x_N)=&tE_\eta(x_1,\ldots,x_N) \nonumber \\
  & \quad +\frac{t x_{k-1}-x_k}{x_{k-1}-x_k} \bigl(E_{\eta}(x_1,\ldots,x_{k-2},x_k,x_{k-1},x_{k+2},\ldots,x_N)-E_{\eta}(x_1,\ldots,x_N) \bigr).
\end{align}
Since $i=k-1$, we have by hypothesis that $E_{\eta}(x_1,\ldots,x_N)\big|_{x_k=0}=0$
and that
\begin{equation}
  E_{s_{k-1} \eta}(x_1,\ldots,x_N)\big|_{x_k=0}=E_{\eta_{k-1}^-}(x_1,\ldots,x_{k-1},x_{k+1},\ldots,x_{N}).
  \end{equation}
Since the right-hand-sides of \eqref{eqrhs1} and \eqref{eqrhs2} have to be equal, we thus get after letting $x_k=0$ that
\begin{equation}
t  E_{\eta}(x_1,\ldots,x_{k-2},0,x_{k-1},x_{k+2},\ldots,x_N) = t E_{\eta_{k-1}^-}(x_1,\ldots,x_{k-1},x_{k+1},\ldots,x_{N})
\end{equation}  
which, after the change of variables $x_{k-1} \longleftrightarrow x_k$, is equivalent to
\begin{equation}
 E_{\eta}(x_1,\ldots,x_N)\big|_{x_{k-1}=0} =  E_{\eta_{k-1}^-}(x_1,\ldots,x_{k-2},x_{k},\ldots,x_{N})\, .
\end{equation}  
Hence \eqref{toprove1} holds when $k$ is replaced by $k-1$.

We now prove that \eqref{toprove2} holds when $k$ is replaced by $k-1$.  Since $i<j=k-1$, we have $\eta_{k-1} \neq 0$ and $\eta_k \neq 0$.  By assumption, we thus obtain
\begin{equation}
  E_{\eta}(x_1,\ldots,x_N)\big|_{x_k=0}=0 \qquad {\rm and} \qquad
  E_{s_{k-1}\eta}(x_1,\ldots,x_N)\big|_{x_k=0}=0 .
\end{equation}
Therefore, after again equating the right-hand-sides of \eqref{eqrhs1} and \eqref{eqrhs2} and letting $x_k=0$, we get
\begin{equation}
t  E_{\eta}(x_1,\ldots,x_{k-2},0,x_{k-1},x_{k+2},\ldots,x_N)=0\, .
\end{equation}  
But this is amounts to  \eqref{toprove2} with $k$ replaced by $k-1$ after 
the change of variables $x_{k-1} \longleftrightarrow x_k$.
\end{proof}

We are now ready to prove the following result.
\begin{proposition}\label{OpColumnaRem2}
  Let $\Lambda$ be a superpartition whose diagram contains a circle in the
  first column, that is, such that $\Lambda_m=0$.  If $\ell=\ell(\Lambda)$, then
\begin{equation}
(-1)^{m-1}\bigl[\partial_{\theta_\ell}P_{\Lambda}(x_1,\ldots,x_\ell;\theta_1,\ldots,\theta_\ell)\bigr]\big|_{x_\ell=0}=P_{\tilde{\mathcal C}\Lambda}(x_1,\ldots,x_{\ell-1};\theta_1,\ldots,\theta_{\ell-1}).
\end{equation}
\end{proposition}
\begin{proof}
By symmetry, it is equivalent to prove
\begin{equation} \label{toproveinpro}
(-1)^{m-1}\bigl[\partial_{\theta_1}P_{\Lambda}(x_1,\ldots,x_\ell;\theta_1,\ldots,\theta_\ell)\bigr]\big|_{x_1=0}=P_{\tilde{\mathcal C}\Lambda}(x_2,\ldots,x_{\ell};\theta_2,\ldots,\theta_{\ell}).
\end{equation}
From \eqref{PvsE}, we get
\begin{equation} \label{eqexpaE}
\partial_{\theta_1} P_\Lambda=c_\Lambda \sum_{\sigma\in S_{2,\ldots,\ell}/S_{2,\ldots,m}\times S_{m^c}}\mathcal{K}_\sigma \theta_2\cdots\theta_m A_m U_{m^c}^+ E_{\Lambda^R} \, ,
\end{equation}
where
\begin{equation}
  c_\Lambda=\frac{(-1)^{\binom{m}{2}}}{f_{\La^s}(t)\, t^{{\rm inv} (\La^s)} }.
\end{equation}
Since $\Lambda=(0,\Lambda_{m-1},\ldots,\Lambda_1,\Lambda_\ell,\ldots,\Lambda_{m+1})$ has only a zero entry in the first position,  we obtain from Lemma~\ref{lemNSZ} that
\begin{equation}
E_{\Lambda^R}(x_1,\ldots,x_\ell)\big|_{x_1=0}=E_{(\Lambda_{m-1},\ldots,\Lambda_1,\Lambda_\ell,\ldots,\Lambda_{m+1})}(x_2,\ldots,x_\ell)= E_{(\tilde {\mathcal C}\Lambda)^R}(x_2,\ldots,x_\ell)
\end{equation}
We can thus deduce from \eqref{eqexpaE} that
\begin{equation}
  [\partial_{\theta_1} P_\Lambda(x_1,\ldots,x_\ell;\theta_1,\ldots,\theta_\ell)]\big|_{x_1=0}= \frac{c_\Lambda}{c_{\tilde {\mathcal C}\Lambda}} P_{\tilde{\mathcal C}\Lambda}(x_2,\ldots,x_\ell;\theta_2,\ldots,\theta_\ell).
\end{equation}
Hence \eqref{toproveinpro} holds since $c_\Lambda/c_{\tilde {\mathcal C}\Lambda}=(-1)^{\binom{m}{2}+{\binom{m-1}{2}}}=(-1)^{m-1}$ given that  $\Lambda^s$ and $(\tilde {\mathcal C}\Lambda)^s$ are equal.
\end{proof}

\section{Evaluation and norm of the Macdonald polynomials in superspace}
\label{sec5}

We prove in this section formulas for the evaluation and norm of the Macdonald polynomials in superspace that were conjectured in \cite{BDLM}.  We will extend the methods used in \cite{lotharingian} to prove the evaluation of the usual Macdonald polynomials.

Let $\Lambda=(\Lambda^a;\Lambda^s)$ be a superpartition of fermionic degree
$m$, and let $\delta_m=(m-1,m-2,\dots,0)$ be the staircase partition.
We define the evaluation ${\mathcal E}^m_{u,q,t}$ on the power-sum basis as
\begin{equation}
{\mathcal E}_{u,q,t}^m(p_\Lambda)=s_{\Lambda^a- \delta_m}\bigl(\frac{1}{q^{m-1}},\frac{t}{q^{m-2}},\ldots,t^{m-1}\bigr) \prod_{i=1}^{\ell(\Lambda^s)}\left[p_{\Lambda_i^s}(\frac{1}{q^{m-1}},\ldots, t^{m-1})+t^{m\Lambda^s_i}\cdot\frac{1-u^{\Lambda^s_i}}{1-t^{\Lambda^s_i}} \right]
\end{equation}
where $u$ is an indeterminate and $s_{\Lambda^a- \delta_m}$ is the Schur function
indexed by the partition $\Lambda^a- \delta_m= (\Lambda_1^a-m+1,\Lambda_2^a-m+2,\dots,\Lambda_m^a)$.  We observe that if $u=t^{N-m}$, then
${\mathcal E}_{t^{N-m},q,t}^m$ is the evaluation considered in \cite{BDLM2} and such that
for $F(x;\theta)=F(x_1,x_2,\dots;\theta_1,\theta_2,\dots)$ a symmetric function in superspace of fermionic degree $m$, we have
\begin{equation} \label{evalvar}
{\mathcal E}_{t^{N-m},q,t}^m\bigl(F(x;\theta)\bigr)=\left[ \frac{\partial_{\theta_1} \cdots \partial_{\theta_m} F(x;\theta)}{\Delta_m(x)}\right]_{x_r=v_r},
\end{equation}
where
\begin{equation}
  v_r=
\left \{ \begin{array}{ll} 
   t^{r-1}/q^{\max(m-r,0)} & {\rm if~} 1\leq r\leq N \\
   0 & {\rm if~} r>N
\end{array} \right .
\end{equation}
and where $\Delta_m(x)$ is the Vandermonde determinant
$\Delta_m(x)=\prod_{1\leq i<j\leq m}(x_i-x_j)$.

For a box $s=(i,j)$ in a partition $\lambda$ (i.e., in row $i$ and column $j$), we introduce the usual arm-lengths and leg-lengths:
\begin{align} 
a_{\lambda}(s)= \lambda_i-j \quad {\rm and} \quad l_{\lambda}(s)= \lambda'_j-i
\end{align}
where we recall that $\lambda'$ stands for the
conjugate of the partition $\lambda$.

Let  $\B(\La)$ denote the set of boxes in the diagram of  $\La$ that do not
appear at the same time in a row containing a circle {and} in a
column containing a circle.

\begin{figure*}[h]\caption{The set of boxes in $\mathcal{B}\Lambda$.}
{\begin{equation*}
\Lambda={\tiny{\tableau[scY]{&&&&&\bl\tcercle{}\\&&& &\\&&&\bl\tcercle{}\\ &&&\bl \\&\bl\tcercle{}\\ &\bl \\ & \bl \\ \bl\tcercle{}}}}\quad\Longrightarrow\quad
\mathcal{B}\Lambda={\tiny{\tableau[scY]{\bl&\bl&&\bl&\\&&&&\\ \bl&\bl& \\ & & \\ \bl\\ &\bl\\  &\bl}}}
\end{equation*}}
\end{figure*}

For $\Lambda$ a superpartition of fermionic degree $m$, let $\S \Lambda$
be the skew diagram $\S \Lambda= \Lambda^{\circledast}/\delta_{m+1}$.

\begin{figure}[h]\caption{The set of boxes in $\mathcal{S}\Lambda$}
{\begin{equation*}
\Lambda={\tiny{\tableau[scY]{&&&&&\bl\tcercle{}\\&&&&\\&&&\bl\tcercle{}\\ &&\bl \\&\bl\tcercle{}\\  &\bl \\ &\bl  \\ \bl\tcercle{}}}}\quad\Longrightarrow\quad
\mathcal{S}\Lambda={\tiny{\tableau[scY]{\bl&\bl&\bl&\bl&&\\\bl&\bl&\bl&&\\ \bl&\bl&&\\\bl &&\bl \\&\\ &\bl \\ &\bl \\ & \bl }}}
\end{equation*}}
\end{figure}

Finally, for a partition $\lambda$, let $n(\lambda)=\sum_{i} (i-1)\lambda_i$.
In the case of a skew partition $\lambda/\mu$,  $n(\lambda/\mu)$ stands for $n(\lambda)-n(\mu)$. 

\begin{theorem}\label{theomain}
Let $\La$ be of fermionic degree $m$.
Then the evaluation formula for the Macdonald polynomials in superspace reads
\begin{equation}
{\mathcal E}_{u,q,t}^m \bigl(P_\La\bigr)= \frac{t^{n(\S \Lambda)+n( (\Lambda')^a/\delta_m)}}
   {q^{(m-1) |\Lambda^a/\delta_m| - n(\Lambda^a/\delta_m)}}
   \frac{\prod_{(i,j)\in \mathcal{S}\Lambda}(1-q^{j-1}t^{m-(i-1)}u)}{\prod_{s\in \mathcal{B}\Lambda} (1-q^{a_{\Lambda^{\circledast}}(s)}t^{l_{\Lambda^*}(s)+1})} 
\end{equation}
or, equivalently, as
\begin{equation}
{\mathcal E}_{u,q,t}^m \bigl(P_\La\bigr)= \frac{t^{n( (\Lambda')^a/\delta_m)}}
   {q^{(m-1) |\Lambda^a/\delta_m| - n(\Lambda^a/\delta_m)}}
   \frac{\prod_{(i,j)\in \mathcal{S}\Lambda}(t^{i-1}-q^{j-1}t^{m}u)}{\prod_{s\in \mathcal{B}\Lambda} (1-q^{a_{\Lambda^{\circledast}}(s)}t^{l_{\Lambda^*}(s)+1})}.  
\end{equation}

\end{theorem}
\newpage
\begin{example}\label{ExEv} Let $\Lambda=(4,1;3)$, the fermionic degree of $\Lambda$ is  $m=2$. We are going to compute ${\mathcal E}_{u,q,t}^m \bigl(P_\La\bigr)$. For this, we draw the diagrams for $\mathcal{S}\Lambda$ and $\mathcal{B}\Lambda$ with their respective contributions in each box:

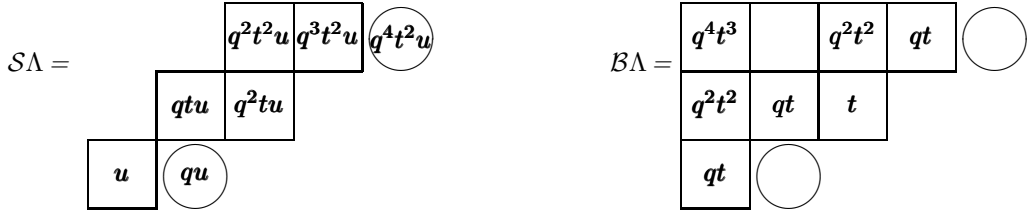
\begin{center}
$\mathcal{S}\Lambda$\,=\,\,\ytableausetup{boxsize=2.5em}
\begin{ytableau}
 \none &\none & { \pmb{q^2 t^2u}} & {\pmb{q^3 t^2u} } &\none[\btcercle{$\pmb{q^4t^2u}$}] \\
  \none & {\pmb{qtu} } & {\pmb {q^2 tu}}\\
{\pmb{u}}  &  \none[\btcercle{$\pmb{qu}$}]
\end{ytableau}
\hspace{2cm}
\ytableausetup{boxsize=2.5em}
$\mathcal{B}\Lambda$\,=\,\,\begin{ytableau}
 {\pmb{q^4t^3}} &{} & {\pmb{q^2t^2}} & {\pmb{qt}} &\none[\btcercle{}]\\
  {\pmb{q^2t^2}} & {\pmb{qt}} & {\pmb{t}}\\
{\pmb{qt}}  &  \none[\btcercle{}]
\end{ytableau}
\end{center}
where $\pmb{a}=1-a$. Moreover it is not easy to see that $\Lambda'=(2,0;3,2,1)$, therefore $n({\Lambda'}^a/\delta_2)=0$. We also have $|\Lambda^a/\delta_2|=4, n(\mathcal{S}\Lambda)=6$ and $n(\Lambda^a/\delta_2)=1$, then
\begin{equation*}
{\mathcal E}_{u,q,t}^m \bigl(P_{(4,1;3)}\bigr)=\frac{t^6}{q^3}\cdot\frac{(1-u)(1-qu)(1-qtu)(1-q^2tu)(1-q^2t^2u)(1-q^3t^2u)(1-q^4t^2u)}{(1-q^4t^3)(1-q^2t^2)^2(1-qt)^3(1-t)}.
\end{equation*}
\end{example}

\begin{remark} \label{remarkzeta}
  The formula for the evaluation conjectured in \cite{BDLM}
  involves a combinatorial  number $\zeta_\La$ instead of $n((\Lambda')^a/\delta_m)$ that we now describe. Consider the partial filling of the squares of $\La$ defined as follows: in each fermionic square of $\La$ write the number of bosonic squares above it; $\zeta_\La$ is obtained by adding up these numbers. Here are three examples for which it is non-vanishing:
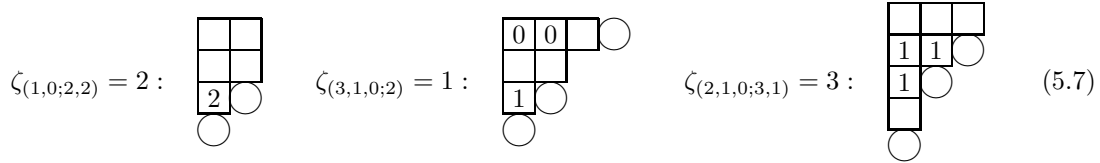
\begin{equation}
\zeta_{(1,0;2,2)}=2:\quad {\tableau[scY]{&\\&\\2&\bl\tcercle{} \\\bl\tcercle{} \\ }}  \qquad \zeta_{(3,1,0;2)}=1:\quad {\tableau[scY]{0&0&&\bl\tcercle{}\\& \\1&\bl\tcercle{} \\\bl\tcercle{} \\ }}  \qquad
\zeta_{(2,1,0;3,1)}=3 :\quad {\tableau[scY]{&&  \\1&1&\bl\tcercle{}  \\1&\bl\tcercle{}\\ \\ \bl\tcercle{} }}\end{equation}
In Lemma~\ref{lemzeta} we will show that  $\zeta_\La$ and $n((\Lambda')^a/\delta_m)$ are equal, thus validating the conjectured expression for the evaluation given in \cite{BDLM}.
\end{remark}

The proof of the theorem is rather non-trivial and will occupy most of the remainder of this section.  To simplify the exposition,
we will establish a few results before actually proceeding to the proof of the theorem.

For a superpartition $\Lambda$ of fermionic degree $m$, we define
\begin{equation}
  \Phi_\Lambda^m(u;q,t) := {\mathcal E}_{u,q,t}^m\bigl(P_\Lambda\bigr)
\end{equation}  
We also let 
\begin{equation}
b_\Lambda(q,t) =  \frac{1}{\|P_\La\|^2}= \frac{1}{\LL  P_{\La}| P_{\La} \RR_{q,t}}.
\end{equation}

\begin{lemma}\label{lema1} We have
\begin{equation} \label{toshow1}
\Phi_\Lambda^m(u;q,t)=(-1)^{|\Lambda|+\binom{m}{2}} t^{-|\Lambda|} \, b_{\Lambda'}(t^{-1},q^{-1})\, \Phi_{\Lambda'}^m(u (qt)^m;t^{-1},q^{-1}).
\end{equation}
\end{lemma}
\begin{proof}  We will first prove that if $F(x;\theta)$ is a symmetric functions in superspace of fermionic degree $m$ and total degree $n$ then
  \begin{equation}\label{Dualidad-Evaluacion}
    {\mathcal E}_{u,q,t}^m \Omega_{q,t}^{-1} F(x;\theta)=  (-1)^{n}q^{-n} {\mathcal E}^m_{u(qt)^m,t^{-1},q^{-1}} F(x;\theta) \, ,
    \end{equation}
  where  $\Omega_{q,t}^{-1}$ is the inverse of the endomorphism $\Omega_{q,t}$ defined in \eqref{defomqt}.  It suffices  to show that \eqref{Dualidad-Evaluacion}
  holds on the power-sum basis.  We have 
\begin{align*}
 {\mathcal E}_{u,q,t}^m \Omega_{q,t}^{-1} \, \bigl( p_\Lambda \bigr) =& (-1)^{|\Lambda|-\ell(\Lambda^s)}\, q^{-|\Lambda^a|}\, \prod_{i=1}^{\ell(\Lambda^s)}\frac{1-t^{\Lambda_i^s}}{1-q^{\Lambda_i^s}}\, {\mathcal E}_{u,q,t}^m \, \bigl( p_\Lambda \bigr)\\
                                                             = &(-1)^{|\Lambda|-\ell(\Lambda^s)}\, q^{-|\Lambda^a|}\, s_{\Lambda^a- \delta_m}\left(\frac{1}{q^{m-1}},\ldots,t^{m-1} \right)  \\
&                                                             \times \prod_{i=1}^{\ell(\Lambda^s)}\frac{1-t^{\Lambda_i^s}}{1-q^{\Lambda_i^s}} \left[ \frac{1}{q^{(m-1)\Lambda_i^s}}\cdot\frac{1-(qt)^{m\Lambda_i^s}}{1-(qt)^{\Lambda_i^s}}+t^{m\Lambda_i^s}\cdot\frac{1-u^{\Lambda_i^s}}{1-t^{\Lambda_i^s}} \right]\\
\end{align*}
We then use the algebraic identity
\begin{equation}
\frac{1-r}{1-s} \left( \frac{s(1-RS)}{S(1-rs)}+ \frac{R}{1-r}\right)= -\frac{1}{s} \left( \frac{R(1-R^{-1}S^{-1})}{r(1-r^{-1}s^{-1})} + \frac{1}{S(1-s^{-1})} \right)
  \end{equation}  
with $r=t^{\Lambda_i^s}, s=q^{\Lambda_i^s}, R=t^{m\Lambda_i^s}$ and $S=q^{m\Lambda_i^s}$ to obtain
\begin{align*}
  {\mathcal E}_{u,q,t}^m \Omega_{q,t}^{-1}\, \bigl( p_\Lambda \bigr)
                                                             =&(-1)^{|\Lambda|-\ell(\Lambda^s)}\, q^{-|\Lambda^a|} s_{\Lambda^a- \delta_m}\left(\frac{1}{q^{m-1}},\ldots,t^{m-1}\right)\\
                                                             &\times \prod_{i=1}^{\ell(\Lambda^s)}\left( -q^{-\Lambda_i^s}\left[t^{(m-1)\Lambda_i^s}\frac{1-(qt)^{-m\Lambda_i^s}}{1-(qt)^{-\Lambda_i^s}}+q^{-m\Lambda_i^s}\cdot \frac{1-u^{\Lambda_i^s}(qt)^{m\Lambda_i^s}}{1-q^{-\Lambda_i^s}}\right]\right)\\
                                                             =&(-1)^{|\Lambda|}q^{-|\Lambda|}  {\mathcal E}_{u(qt)^m,t^{-1},q^{-1}}^m \bigl( p_\Lambda \bigr),
\end{align*}
where we used the relation $|\Lambda|=|\Lambda^a|+|\Lambda^s|$.
This proves \eqref{Dualidad-Evaluacion}.

Now, using Theorem~\ref{theoduality}, we have
\begin{equation}
P_\Lambda(q,t)=(-1)^{\binom{m}{2}}(qt^{-1})^{|\Lambda|} b_{\Lambda'}(t^{-1},q^{-1}) \, \Omega_{q,t}^{-1} \, P_{\Lambda'}(t^{-1},q^{-1}).
\end{equation}
Therefore, applying ${\mathcal E}_{u,q,t}^m$ on both sides of the previous equation and using \eqref{toshow1}, we get
\begin{equation}
  {\mathcal E}^m_{u,q,t} \, \bigl(P_{\Lambda}(q,t) \bigr) =(-1)^{|\Lambda|+\binom{m}{2}} t^{-|\Lambda|}
   \, b_{\Lambda'}(t^{-1},q^{-1})\, {\mathcal E}^m_{u(qt)^m,t^{-1},q^{-1} } \, \bigl( P_{\Lambda'}(t^{-1},q^{-1}) \bigr),
\end{equation}
which proves the lemma.
\end{proof}

We can now give $\Phi^m_\Lambda(u;q,t)$ explicitly up to a constant in $\mathbb Q(q,t)$.

\begin{lemma}\label{Evaluacion11}
We have
\begin{align*}
\Phi^m_\Lambda(u;q,t)=v_\Lambda(q,t) \prod_{(i,j)\in \mathcal{S}\Lambda}\left(1-q^{j-1}t^{m-(i-1)}u\right),
\end{align*}
where $v_\Lambda(q,t)  \in \mathbb Q(q,t)$.
\end{lemma}
\begin{proof}
It is known \cite{} that
\begin{equation}
\frac{\partial_{\theta_1} \cdots \partial_{\theta_m} m_\Gamma(x_1,\ldots,x_N)}{\Delta_m(x)}=s_{\Gamma^a-\delta_m}(x_1,\ldots,x_m) m_{\Gamma^s}(x_{m+1},\ldots,x_N),
\end{equation}
which implies that if $N-m<\ell(\Gamma^s)$ then $m_{\Gamma^s}(x_{m+1},\ldots,x_N)=0$. Now, by triangularity, we have
\begin{equation}
P_\Lambda=m_\Lambda+\sum_{\Omega<\Lambda} c_\Omega (q,t) \, m_\Omega ,
\end{equation}
where by definition of the dominance order on superpartitions the superpartitions $\Omega$ that appear in the sum are such that $\ell(\Omega^s) \geq \ell(\Lambda^s)$.  Hence
\begin{equation}
\frac{\partial_{\theta_1} \cdots \partial_{\theta_m} P_\Lambda(x_1,\ldots,x_N)}{\Delta_m(x)}= 0
\end{equation}
whenever $N-m<\ell(\Lambda^s)$.  We can thus use \eqref{evalvar} to conclude that 
\begin{equation}
\Phi_\Lambda^m(u;q,t)=0 \qquad {\rm if}\quad u=1,t,\ldots,t^{\ell(\Lambda^s)-1}. 
\end{equation}
Using Lemma~\ref{lema1}, this implies
\begin{equation}
\Phi_{\Lambda'}^m(u (qt)^m;t^{-1},q^{-1})=0 \qquad {\rm if}\quad u=1,t,\ldots,t^{\ell(\Lambda^s)-1},
\end{equation}
which amounts to
\begin{equation}
\Phi_{\Lambda}^m(u ;q,t)=0 \qquad {\rm if}\quad u=q^{-m}t^{-m}, q^{-(m+1)}t^{-m},\ldots, q^{1-\Lambda_1^{\circledast}}t^{-m} 
\end{equation}
since $\ell\bigr((\Lambda')^s\bigl)+m=\Lambda_1^{\circledast}$.  It is thus immediate that
\begin{align}\label{divideA}
  \prod_{i=m+1}^{\Lambda_1^{\circledast}} (1-q^{i-1}t^m u) \quad {\rm divides} \quad
  \Phi_\Lambda^m(u;q,t) \quad {\rm in}\quad  \mathbb Q(q,t)[u] \, .
\end{align}

From Proposition~\ref{PropSkew}, we have
\begin{align} \label{skeweq}
P_{\Lambda} (x;\theta)=\sum_\Lambda \frac{1}{\|P_\Omega\|^2}P_{\Lambda/\Omega} (x_1,\theta_1) \, P_{\Omega} (x_2,x_3,\dots;\theta_2,\theta_3,\dots), 
\end{align}
Since $P_{\Gamma} (x_1,\theta_1)=0$ whenever $\ell(\Gamma)>1$, we have
that if $P_{\Lambda/\Omega} (x_1,\theta_1)\neq 0$ then either $P_{\Lambda/\Omega} =P_{(\emptyset;r)}$ or $P_{\Lambda/\Omega}=P_{(r;)}$ for a certain $r$.
From Propositions \ref{PropgI} and \ref{PropgII}, this implies that $P_{\Lambda/\Omega} (x_1,\theta_1)=0$ unless $\Lambda/\Omega$ is a horizontal $r$-strip or a
horizontal $\tilde r$-strip.  Consequently, 
 $\partial_{\theta_1}P_{\Lambda/\Omega} (x_1,\theta_1)=0$ unless $\Lambda/\Omega$ is  a
horizontal $\tilde r$-strip. Therefore, we deduce from \eqref{skeweq} that  
\begin{align} \label{skeweval}
{\mathcal E}_{t^{N-m},q,t}^m \bigl(P_\Lambda\bigr)=\sum_\Omega \psi_{\Lambda/\Omega}(q,t) \, {\mathcal E}_{t^{N-1-(m-1)},q,t}^{m-1}\bigl (P_\Omega \bigr).
\end{align}
where the sum is over all $\Omega$'s such that $\Lambda/\Omega$ is a horizontal $\tilde r$-strip for a certain $r$.  Observe that in the previous equation, the coefficient
$\psi_{\Lambda/\Omega}(q,t)$ can be given explicitly: 
\begin{equation} \psi_{\Lambda/\Omega}(q,t) = 
\frac{(-1)^{m-1} t^{|\Omega|}}{\|P_\Omega\|^2}\left[\frac{\partial_{\theta_1}P_{\Lambda/\Omega}(x_1;\theta_1)}{\prod_{1\leq j\leq m}(x_1-x_j)} \right]_{x_r=v_r} 
\end{equation}  
where the sign comes from the commutation of $\partial_{\theta_2} \cdots \partial_{\theta_m}$ with $P_{\Lambda/\Omega}(x_1;\theta_1)$, which is of fermionic degree 1, 
and $t^{|\Omega|}$ comes from the fact that $P_\Omega$ is of total degree $|\Omega|$ and that the new evaluation of each variable $x_2,x_3,\dots,x_N$ needs to be multiplied by $t$ to coincide with the original evaluation.
Now, given that \eqref{skeweval} holds for $N=m,m+1,m+2,\dots$, we have that
\begin{align} \label{skewevalbest}
\Phi_\Lambda^m(u;q,t)=\sum_\Omega \psi_{\Lambda/\Omega}(q,t) \, \Phi_\Omega^{m-1}(u;q,t) 
\end{align}
where the sum is over all $\Omega$'s such that $\Lambda/\Omega$ is a horizontal $\tilde r$-strip for a certain $r$.

Since $\Lambda/\Omega$ is a horizontal $\tilde r$-strip, we have that
$\Omega^\cd \supseteq (\Lambda_2^\cd,\Lambda_3^\cd,\dots)$.  Hence, from
\eqref{divideA}, we see that $\prod_{i=m}^{\Lambda_2^{\circledast}}(1-q^{i-1}t^{m-1}u)$ always divides $\Phi_\Omega^{m-1}(u;q,t)$ in \eqref{skewevalbest}, which gives that
\begin{align}
  \prod_{i=m}^{\Lambda_2^{\circledast}} (1-q^{i-1}t^m u) \quad {\rm divides} \quad
  \Phi_\Lambda^m(u;q,t) \quad {\rm in}\quad  \mathbb Q(q,t)[u] \, .
\end{align}
Repeating the argument again and again, we get that
\begin{equation} \label{divideAA}
  \prod_{j=1}^m \prod_{i=m-j+2}^{\Lambda^\circledast_j}(1-q^{i-1}t^{m-(j-1)}u)
  \quad {\rm divides} \quad
  \Phi_\Lambda^m(u;q,t) \quad {\rm in}\quad  \mathbb Q(q,t)[u] 
\end{equation}  
and that
\begin{align} \label{skewevalbestnew}
\Phi_\Lambda^m(u;q,t)=\sum_\Omega \phi_{\Lambda/\Gamma}(q,t) \, \Phi_\Gamma^{0}(u;q,t) 
\end{align}
for some coefficients $\phi_{\Lambda/\Gamma}(q,t)$, 
where the sum is over superpartition such that $\Gamma^\cd \supseteq (\Lambda_{m+1}^\cd,\Lambda_{m+2}^\cd,\dots)$.  Hence, using Theorem 5.3 of \cite{lotharingian}, we have that every term $\Phi_\Gamma^{0}(u;q,t)$ in \eqref{skewevalbestnew} is divisible by 
\begin{equation*}
 \prod_{j=m+1}^{\ell(\Lambda^\circledast)}\prod_{i=1}^{\Lambda^\circledast_j}(1-q^{i-1}t^{m-(j-1)}u) 
\end{equation*}  
and, consequently, so is $\Phi_\Lambda^m(u;q,t)$.  Together with \eqref{divideAA}, this implies that 
\begin{align} \label{lastdivi}
\prod_{(i,j)\in \mathcal{S}\Lambda}(1-q^{j-1}t^{m-(i-1)}u)  \quad {\rm divides} \quad
  \Phi_\Lambda^m(u;q,t) \quad {\rm in}\quad  \mathbb Q(q,t)[u] \, . 
\end{align}
Now, ${\mathcal E}_{u,q,t}^m(p_\Omega)$ is a polynomial in 
$u$ of degree $|\Omega^s|$.  For $\Omega$'s of fixed fermionic and total degrees,
$|\Omega^s|$ is maximal when $\Omega^a=\delta_m$, in which case 
$|\Omega^s|=|\Omega|-|\delta_m|$.  Therefore, ${\mathcal E}_{u,q,t}^m(P_\Lambda)=\Phi_\Lambda^m(u;q,t)$
is a polynomial in $u$ of degree at most  $|\Lambda|-|\delta_m|$. From \eqref{lastdivi},
we conclude immediately that
\begin{align}\label{Evaluacion1}
\Phi_\Lambda^m(u;q,t)=v_\Lambda(q,t) \prod_{(i,j)\in \mathcal{S}\Lambda}(1-q^{j-1}t^{m-(i-1)}u).
\end{align}
since $\prod_{(i,j)\in \mathcal{S}\Lambda}(1-q^{j-1}t^{m-(i-1)}u)$ is a polynomial in
$u$ of degree  $|\Lambda^\cd|-|\delta_{m+1}|=|\Lambda|-|\delta_m|$.
\end{proof}
The coefficients $v_\Lambda(q,t)$ in Lemma~\ref{Evaluacion11} satisfy the
following recursion.
\begin{lemma}\label{Relacionvcv}
  Let $\Lambda$ be a superpartition of fermionic degree $m$ such that $\Lambda_m\neq 0$ and let
  $\ell=\ell(\Lambda)$. Then
\begin{align}
\frac{v_\Lambda(q,t)}{v_{\mathcal{C}\Lambda}(q,t)}=\displaystyle\frac{t^{\binom{\ell}{2}}}{q^{\binom{m}{2}}}\cdot \frac{1}{\prod_{i=1}^{\ell}1-q^{\Lambda_i^{\circledast}-1}t^{\ell-(i-1)}}.
\end{align}
where we recall that $\mathcal C \Lambda$ is the superpartition whose diagram is that of $\Lambda$ without its first column.
\end{lemma}
\begin{proof}
Given that $\Lambda_m\neq 0$, Proposition~\ref{OpColumnaRem} gives
\begin{align} \label{premiere}
P_\Lambda(x_1,\ldots,x_\ell;\theta_1,\ldots,\theta_\ell)=x_1\cdots x_\ell P_{\mathcal{C}\Lambda}(x_1,\ldots,x_\ell;\theta_1,\ldots,\theta_\ell).
\end{align}
As we have seen, ${\mathcal E}_{t^{\ell-m},q,t}$ corresponds to the evaluation when the number of variables is equal to $\ell$.  Recalling that $\Phi_\Lambda(t^{\ell-m};q,t)={\mathcal E}_{t^{\ell-m},q,t}(P_\Lambda)$, we have by applying ${\mathcal E}_{t^{\ell-m},q,t}$ on both sides
of \eqref{premiere} that
\begin{align}
\Phi_\Lambda(t^{\ell-m};q,t)=\frac{t^{\binom{\ell}{2}}}{q^{\binom{m}{2}}}
\Phi_{\mathcal{C}\Lambda}(t^{\ell-m};q,t)
\end{align}
Using Lemma~\ref{Evaluacion11}, we can immediately deduce that
\begin{align}
\frac{v_\Lambda(q,t)}{v_{\mathcal{C}\Lambda}(q,t)}=\displaystyle\frac{t^{\binom{\ell}{2}}}{q^{\binom{m}{2}}}\cdot \frac{\prod_{(i,j)\in \mathcal{S}(\mathcal C\Lambda)} 1-q^{j-1}t^{\ell-(i-1)}}{\prod_{(i,j)\in \mathcal{S}\Lambda}1-q^{j-1}t^{\ell-(i-1)}}= \displaystyle\frac{t^{\binom{\ell}{2}}}{q^{\binom{m}{2}}}\cdot \frac{1}{\prod_{i=1}^{\ell}1-q^{\Lambda_i^{\circledast}-1}t^{\ell-(i-1)}}.
\end{align}
\end{proof}

In order to find the explicit value of $v_\Lambda(q,t)$, we need another recursion when the first column of the diagram of $\Lambda$ contains a circle.  This turns out to be a little tricky.  For that purpose, following what was done in the case of the Jack polynomials in superspace \cite{DLMeva}, we first need to define a second type of evaluation.  Let  $\widetilde{{\mathcal E}}^m_{u,q,t}$ be such that on a symmetric function in superspace $F(x;\theta)$ of
fermionic degree $m$ and total degree $N$
\begin{align*}
  \widetilde{{\mathcal E}}^m_{u,q,t}\bigl(F(x:\theta)\bigr)= {\mathcal E}_{u;q,t}^{m-1} \Bigl[\bigl[\partial_{\theta_1} F(x;\theta)\bigr]_{x_1=0} \Bigr]_{x_i\rightarrow x_{i-1},\theta_i \rightarrow
    \theta_{i-1} }, 
\end{align*}
where $x_i\rightarrow x_{i-1}, \theta_i\rightarrow \theta_{i-1},$ means that the set of variables $(x_2,x_3,\dots; \theta_2, \theta_3,\dots)$ is sent to
$(x_1,x_2,\dots; \theta_1, \theta_2,\dots)$. Note that when $u=t^{N-m}$
the second evaluation takes by symmetry the simpler form
\begin{align}
  \widetilde{{\mathcal E}}_{t^{N-m},q,t}\bigl(F(x,\theta)\bigr):={\mathcal E}^{m-1}_{t^{N-m},q,t}
  \bigl(\partial_{\theta_N} F(x;\theta)\bigr)
\end{align}
since $x_N=0$ in the evaluation ${\mathcal E}^{m-1}_{t^{N-m},q,t}$.
For a superpartition $\Lambda$ of fermionic degree $m$, we define
\begin{equation}
 \widetilde{\Phi}_\Lambda^m(u;q,t) := \widetilde{{\mathcal E}}_{u,q,t}^m\bigl(P_\Lambda\bigr)
\end{equation}  
We also let $\widetilde{S}\Lambda$ be the skew partition
\begin{equation}
  \widetilde{S}\Lambda:=\Lambda^*/\delta_{m} \, .
\end{equation}

\begin{theorem}\label{theomain2}
Let $\La$ be of fermionic degree $m>0$.
Then the second evaluation formula for the Macdonald polynomials in superspace reads
\begin{equation}
  \widetilde{{\mathcal E}}_{u,q,t}^m
  \bigl(P_\La\bigr)= \frac{t^{n(\widetilde{\S} \Lambda)+n( (\Lambda')^a/\delta_{m-1})}}
   {q^{(m-2) |\Lambda^a/\delta_{m-1}| - n(\Lambda^a/\delta_{m-1})}}
   \frac{\prod_{(i,j)\in \widetilde{\mathcal{S}}\Lambda}(1-q^{j-1}t^{m-i}u)}{\prod_{s\in \mathcal{B}\Lambda} (1-q^{a_{\Lambda^{\circledast}}(s)}t^{l_{\Lambda^*}(s)+1})}
\end{equation}
or, equivalently, as
\begin{equation}
  \widetilde{{\mathcal E}}_{u,q,t}^m
  \bigl(P_\La\bigr)= \frac{t^{n( (\Lambda')^a/\delta_{m-1})}}
   {q^{(m-2) |\Lambda^a/\delta_{m-1}| - n(\Lambda^a/\delta_{m-1})}}
   \frac{\prod_{(i,j)\in \widetilde{\mathcal{S}}\Lambda}(t^{i-1}-q^{j-1}t^{m-1}u)}{\prod_{s\in \mathcal{B}\Lambda} (1-q^{a_{\Lambda^{\circledast}}(s)}t^{l_{\Lambda^*}(s)+1})}
\end{equation}
\end{theorem}
\begin{proof}
Theorem~\ref{theomain2} will follow immediately from Lemma~\ref{skewtilde}
once the explicit expression for $\widetilde{v}_\Lambda(q,t)$ 
has been established in Corollary~\ref{coronouv}.
\end{proof}

\begin{example}
  We are going to calculate $\widetilde{\mathcal E}_{u,q,t}^m \bigl(P_\La\bigr)$ for $\Lambda=(4,1;3)$ (the superpartition used in 
  Example~\ref{ExEv}).
  For this, we dras the diagrams $\widetilde{\mathcal{S}}\Lambda$ y $\mathcal{B}\Lambda$ and with their respective factors in each box: 

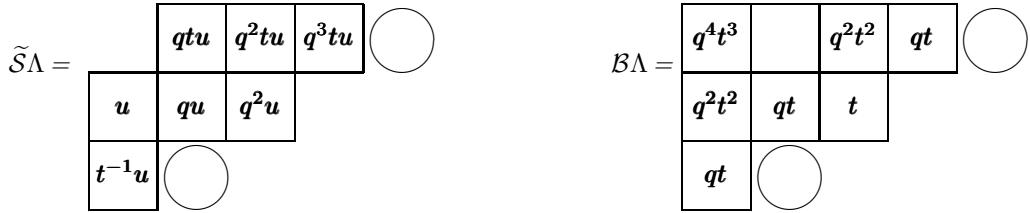
\begin{center}
$\widetilde{\mathcal{S}}\Lambda\,$=\,\,\ytableausetup{boxsize=2.5em}
\begin{ytableau}
 \none & {\pmb{qtu}} & {\pmb{q^2tu}} & {\pmb{q^3 t u} } &\none[\btcercle{{}}] \\
 {\pmb u}& {\pmb{qu}}&{\pmb{q^2u}} \\
{\pmb{t^{-1}u}}  &  \none[\btcercle{}]
\end{ytableau}
\hspace{2cm}
\ytableausetup{boxsize=2.5em}
$\mathcal{B}\Lambda\,$=\,\,\begin{ytableau}
 {\pmb{q^4t^3}} &{} & {\pmb{q^2t^2}} & {\pmb{qt}} &\none[\btcercle{}]\\
  {\pmb{q^2t^2}} & {\pmb{qt}} & {\pmb{t}}\\
{\pmb{qt}}  &  \none[\btcercle{}]
\end{ytableau}
\end{center}
where $\pmb{a}=1-a$. We also have  $n({\Lambda'}^a/\delta_1)=0$ since $\Lambda'=(2,0;3,2,1)$. Moreover, $|\Lambda^a/\delta_1|=5, n(\widetilde{\mathcal{S}}\Lambda)=5$ and $n(\Lambda^a/\delta_1)=1$. Hence
\begin{equation*}
\widetilde{\mathcal E}_{u,q,t}^m \bigl(P_{(4,1;3)}\bigr)=qt^{5}\cdot\frac{(1-u)(1-t^{-1}u)(1-qu)(1-q^2u)(1-qtu)(1-q^2tu)(1-q^3tu)}{(1-q^4t^3)(1-q^2t^2)^2(1-qt)^3(1-t)}.
\end{equation*}

\end{example}

We can again give $ \widetilde{{\mathcal E}}_{u,q,t}^m
  \bigl(P_\La\bigr)=\widetilde{\Phi}^m_\Lambda(u;q,t)$ explicitly up to a constant in $\mathbb Q(q,t)$.
\begin{lemma} \label{skewtilde}
 We have
\begin{align}
\widetilde{\Phi}^m_\Lambda(u;q,t)=\widetilde{v}_\Lambda(q,t) \prod_{(i,j)\in\widetilde{\mathcal{S}}\Lambda} (1-q^{j-1}t^{m-1-(i-1)}u) ,
\end{align} 
where $\widetilde{v}_\Lambda(q,t) \in \mathbb Q(q,t)$.
\end{lemma}
\begin{proof}
We use again Proposition~\ref{PropSkew} to get
\begin{align} \label{skeweq2}
P_{\Lambda} (x;\theta)=\sum_\Lambda \frac{1}{\|P_\Omega\|^2}P_{\Lambda/\Omega} (x_1,\theta_1) \, P_{\Omega} (x_2,x_3,\dots;\theta_2,\theta_3,\dots), 
\end{align}
As seen in the proof of Lemma~\ref{Evaluacion11},
$\partial_{\theta_1}P_{\Lambda/\Omega} (x_1,\theta_1)=0$ unless $\Lambda/\Omega$ is  a
horizontal $\tilde r$-strip, which implies that $\bigl[\partial_{\theta_1}P_{\Lambda/\Omega} (x_1,\theta_1)\bigr]_{x_1=0}=0$ unless $\Lambda/\Omega$ is  a
horizontal $\tilde 0$-strip, that is, unless $\Omega$ can be obtained by removing a circle from $\Lambda$.

Therefore, we deduce from \eqref{skeweq2} that  
\begin{align} \label{skeweval2}
\widetilde{{\mathcal E}}_{t^{N-m},q,t}^m \bigl(P_\Lambda\bigr)=\sum_\Omega \widetilde{\psi}_{\Lambda/\Omega}(q,t) \, {\mathcal E}_{t^{N-1-(m-1)},q,t}^{m-1}\bigl (P_\Omega \bigr).
\end{align}
where the sum is over all $\Omega$'s such that $\Lambda/\Omega$ is a horizontal $\tilde 0$-strip.  The coefficient
$\psi_{\Lambda/\Omega}(q,t)$ can this time be given explicitly as 
\begin{equation} \widetilde{\psi}_{\Lambda/\Omega}(q,t) = 
\frac{(-1)^{m-1}}{\|P_\Omega\|^2}\left[\frac{\partial_{\theta_1}P_{\Lambda/\Omega}(x_1;\theta_1)}{\prod_{1\leq j\leq m}(x_1-x_j)} \right]_{x_1=0,x_2=v_1,\dots,x_m=v_{m-1}} 
\end{equation}  
where the sign comes from the commutation of $\partial_{\theta_2} \cdots \partial_{\theta_m}$ with $P_{\Lambda/\Omega}(x_1;\theta_1)$, which is of fermionic degree 1. 
Again, given that \eqref{skeweval2} holds for $N=m,m+1,m+2,\dots$, we have that
\begin{align} \label{skewevalbest2}
\widetilde{\Phi}_\Lambda^m(u;q,t)=\sum_\Omega \widetilde{\psi}_{\Lambda/\Omega}(q,t) \, \Phi_\Omega^{m-1}(u;q,t) 
\end{align}
where the sum is over all $\Omega$'s such that $\Lambda/\Omega$ is a horizontal $\tilde 0$-strip.  

Since $\Omega$ is obtained from $\Lambda$ by removing a circle, we have
that $\Omega$ is a superpartition of fermionic degree $m-1$ such that
$\Omega^\cd \supseteq \Omega^*= \Lambda^*$.  Hence,
from Lemma~\ref{Evaluacion11}, every term $\Phi_\Omega^{m-1}(u;q,t)$ in 
\eqref{skewevalbest2} is divisible in $\mathbb Q(q,t)[u]$ by 
\begin{align*}
\prod_{(i,j)\in\widetilde{\mathcal{S}}\Lambda} (1-q^{j-1}t^{m-1-(i-1)}u) 
\end{align*}
and thus so is $\widetilde{\Phi}_\Lambda^m(u;q,t)$.

The evaluation $\widetilde{{\mathcal E}}_{u,q,t}^m(p_\Omega)$ is again a polynomial in 
$u$ of degree $|\Omega^s|$ which means that
$\widetilde{{\mathcal E}}_{u,q,t}^m(P_\Lambda)=\widetilde{\Phi}_\Lambda^m(u;q,t)$
is still a polynomial in $u$ of degree at most  $|\Lambda|-|\delta_m|$.
Thus, from the previous observation, 
\begin{align}\label{Evaluacion2}
\widetilde{\Phi}^m_\Lambda(u;q,t)=\widetilde{v}_\Lambda(q,t)\prod_{(i,j)\in\widetilde{\mathcal{S}}\Lambda} (1-q^{j-1}t^{m-1-(i-1)}u).
\end{align}
since $\prod_{(i,j)\in\widetilde{\mathcal{S}}\Lambda} (1-q^{j-1}t^{m-1-(i-1)}u)$
is a polynomial in
$u$ of degree  $|\Lambda|-|\delta_m|$.
\end{proof}

Remarkably, the coefficients $v_\Lambda(q,t)$ and $\widetilde{v}_\Lambda(q,t)$ in
Lemmas~\ref{Evaluacion11} and \ref{skewtilde} are equal up to powers of $q$ and $t$.
\begin{lemma}\label{relacionE1E2}
  We have
\begin{align}
  \frac{{v}_\Lambda(q,t)}{\widetilde{v}_\Lambda(q,t)}=
  \frac{t^{|(\Lambda')^a|-(m-1)^2}}{q^{|\Lambda^a|-(m-1)^2}}
\end{align}
\end{lemma}
\begin{proof}
  As was mentioned before, for all $\Omega$'s of fixed fermionic and total degrees,  $\widetilde{{\mathcal E}}_{u,q,t}^m(p_\Omega)$ is of
maximal degree in $u$ when $\Omega$ is of the form $\Omega=(\delta_m;\Omega^s)$.
By the triangularity between the power-sums and monomial bases \cite{}, this is also the case for the monomial basis, that is, $\widetilde{{\mathcal E}}_{u,q,t}^m(m_\Omega)$ is of
maximal degree in $u$ when $\Omega$ is of the form $\Omega=(\delta_m;\Omega^s)$.
In that case, we have
\begin{equation*}
\frac{\partial_{\theta_1}\cdots \partial_{\theta_{m-1}} \partial_{\theta_N} m_\Omega }{\Delta_{m-1}(x) \prod_{1\leq i\leq m-1} (x_i-x_N)} \prod_{1\leq i\leq m-1} (x_i-x_N)= \left[ \prod_{1\leq i\leq m-1} (x_i-x_N) \right] m_{\Omega^s}(x_m,x_{m+1},\dots,x_{N-1})
\end{equation*}  
Therefore,  considering the evaluation in $N-1$ variables (in which case $x_N=0$),  we get
\begin{align*}
\widetilde{\mathcal E}^m_{t^{N-m},q,t}(m_\Omega)&={\mathcal E}^{m-1}_{t^{N-1-(m-1)},q,t}\bigl(\partial_{\theta_N} m_\Omega \bigr)
\\
&=\left(\frac{1}{q^{m-2}}\right)  \left(\frac{t}{q^{m-3}}\right)  \cdots \left(\frac{t^{m-2}}{1}\right)  m_{\Omega^s}(t^{m-1},\ldots,t^{N-2})\\
&=\left(\frac{t}{q}\right)^{\binom{m-1}{2}} t^{- |\Omega^s|} m_{\Omega^s}(t^{m},\ldots,t^{N-1})
\end{align*} 
which implies that
\begin{align}  \label{great}
\widetilde{\mathcal E}^m_{t^{N-m},q,t}(m_\Omega)=\left(\frac{t}{q}\right)^{\binom{m-1}{2}} t^{-|\Omega^s|} {\mathcal E}^m_{t^{N-m},q,t}(m_\Omega), 
\end{align}
since in our case
\begin{equation*}
\frac{\partial_{\theta_1}\cdots  \partial_{\theta_m} m_\Omega (x_1,\dots,x_{N};\theta_1,\dots,\theta_{N}) }{\Delta_{m}(x)} = m_{\Omega^s}(x_{m+1},\dots,x_{N}) .
\end{equation*}
Now, given that \eqref{great} is valid for all values of $N$, we can conclude that
\begin{align}  \label{great2}
\widetilde{\mathcal E}^m_{u,q,t}(m_\Omega)=\left(\frac{t}{q}\right)^{\binom{m-1}{2}} t^{-|\Omega^s|} {\mathcal E}^m_{u,q,t}(m_\Omega).  
\end{align}
From Lemma~\ref{skewtilde}, $\widetilde{\mathcal E}^m_{u,(q,t)}(P_\Lambda)$ is a polynomial
of degree $|\Lambda|-|\delta_m|=|\Lambda^*|-\binom{m}{2}$.  As we have seen, this degree is only obtainable from the $m_\Omega$'s in the expansion of $P_\Lambda$ such that $\Omega$ is of the form $\Omega=(\delta_m;\Omega^s)$ with $|\Omega^s|
=|\Lambda^*|-\binom{m}{2}$.  Therefore, from  \eqref{great2}, we get that
when considering only the maximal coefficient
\begin{align}
  \widetilde{\mathcal E}^m_{u,q,t}(P_\Lambda)\big|_{u^{\max}}
=\left(\frac{t}{q}\right)^{\binom{m-1}{2}} t^{-|\Lambda^*|+|\delta_m|}  {\mathcal E}^m_{u,q,t}(P_\Lambda)\big|_{u^{\max}}.
\end{align}
From Lemmas~\ref{Evaluacion11} and \ref{skewtilde}, we thus get that
\begin{align}
\left(\frac{t}{q}\right)^{\binom{m-1}{2}} t^{-|\Lambda^*|+|\delta_m|} v_\Lambda(q,t) \prod_{(i,j)\in \mathcal{S}\Lambda} q^{j-1} t^{m-(i-1)} = \widetilde{v}_\Lambda(q,t) \prod_{(i,j)\in \widetilde{\mathcal{S}}\Lambda} q^{j-1}t^{m-1-(i-1)}, 
\end{align}
which amounts to
\begin{align}
\frac{v_\Lambda(q,t)}{\widetilde{v}_\Lambda(q,t)} =  \left(\frac{t}{q}\right)^{-\binom{m-1}{2}}  \frac{ \prod_{(i,j)\in \widetilde{\mathcal{S}}\Lambda} q^{j-1}t^{m-(i-1)}}{  \prod_{(i,j)\in \mathcal{S}\Lambda} q^{j-1} t^{m-(i-1)}}  = \left(\frac{t}{q}\right)^{-\binom{m-1}{2}}  \frac{ \prod_{(i,j)\in \widetilde{\mathcal{S}}\Lambda} q^{j-1}t^{-(i-1)}}{  \prod_{(i,j)\in \mathcal{S}\Lambda} q^{j-1} t^{-(i-1)}}  
\end{align}
since $ \mathcal{S}\Lambda$ and  $\widetilde{\mathcal{S}}\Lambda$ have the same number of cells.  For the cells of $\mathcal{S}\Lambda$ that do not belong to
$\widetilde{\mathcal{S}}\Lambda$, we have that $i-1$ (resp. $j-1$) is the length of a given column (resp. row) of $\Lambda^a$. This explains the factor $t^{|(\Lambda')^a|}/q^{|\Lambda^a|}$ in \eqref{relacionE1E2}.  Finally, the cells of
$\widetilde{\mathcal{S}}\Lambda$ that do not belong to $\mathcal{S}\Lambda$
provide a factor $(t/q)^{-\binom{m}{2}}$.  The lemma then follows since
$\binom{m}{2}+\binom{m-1}{2}=(m-1)^2$.
\end{proof}
We can now establish our second recursion for $v_\Lambda(q,t)$ that will apply when the
first column in the diagram of $\Lambda$ contains a circle.
\begin{lemma} \label{relacionvctv}  If $\Lambda$ is a superpartition
of length $\ell$  such that $\Lambda_m=0$, then
\begin{align}
  \frac{v_\Lambda(q,t)}{v_{\widetilde{\mathcal C}\Lambda}(q,t)}=
\frac{t^{|(\Lambda')^a|-(m-1)^2}}{q^{|\Lambda^a|-(m-1)^2}}
\prod_{i \in fr(\tilde {\mathcal C} \Lambda)}(1-q^{\Lambda_i^{\circledast}-1}t^{\ell-1-(i-1)}),
\end{align}
where $fr(\tilde {\mathcal C} \Lambda)$ stands for the rows of $\tilde {\mathcal C}\Lambda$ that contain a circle.
\end{lemma}
\begin{proof}
  From our hypotheses, we can use Proposition~\ref{OpColumnaRem2} to get 
\begin{align}\label{derivadaN}
\bigl[ \partial_\ell P_\Lambda(x_1,\ldots,x_\ell;\theta_1,\ldots,\theta_\ell)\bigr]_{x_\ell=0}=P_{\widetilde{\mathcal{C}}\Lambda}(x_1,\ldots,x_{\ell-1};\theta_1,\ldots,\theta_{\ell-1}).
\end{align}
Applying the evaluation in $\ell-1$ variables ${\mathcal E}^{m-1}_{t^{\ell-1-(m-1)},q,t}$
on both sides of the previous equation, we obtain
\begin{align}
\widetilde{\Phi}_\Lambda^m(t^{\ell-m};q,t)=\Phi_{\widetilde{\mathcal{C}}\Lambda}^{m-1}(t^{\ell-m};q,t) .
\end{align}
Therefore, from Lemmas~\ref{Evaluacion11} and \ref{skewtilde}, we have
\begin{align}
 \widetilde{v}_\Lambda(q,t) \prod_{(i,j)\in\widetilde{\mathcal{S}}\Lambda} (1-q^{j-1}t^{\ell-1-(i-1)}) = v_{\tilde{\mathcal C}\Lambda}(q,t) \prod_{(i,j) \in \mathcal{S}(\tilde{\mathcal C}\Lambda)}(1-q^{j-1}t^{\ell-1-(i-1)})
\end{align}
which gives
\begin{align}\label{vtildev}
\frac{\widetilde{v}_\Lambda(q,t)}{v_{\widetilde{\mathcal C}\Lambda}(q,t)}=\prod_{i \in fr(\tilde{\mathcal C}\Lambda)}(1-q^{\Lambda_i^{\circledast}-1}t^{\ell-1-(i-1)}).
\end{align}
The lemma then follows from Lemma~\ref{relacionE1E2}.
\end{proof}

We can now proceed to the proof of Theorem~\ref{theomain}.
\begin{proof}[Proof of Theorem~\ref{theomain}]  By Lemma ~\ref{Evaluacion11}, we have 
\begin{align*}
{\mathcal E}_{u,q,t}^m\bigl(P_\Lambda(q,t)\bigr)=v_\Lambda(q,t) \prod_{(i,j)\in \mathcal{S}\Lambda}\left(1-q^{j-1}t^{m-(i-1)}u\right)
\end{align*}
where $v_\Lambda(q,t)\in \mathbb{Q}(q,t)$.  The theorem will thus follow if we can show that
\begin{equation} \label{vvaprouver}
v_\Lambda(q,t) = \frac{t^{n(\S \Lambda)+n( (\Lambda')^a/\delta_m)}}
   {q^{(m-1) |\Lambda^a/\delta_m| - n(\Lambda^a/\delta_m)}}
   \frac{1}{\prod_{s\in \mathcal{B}\Lambda} (1-q^{a_{\Lambda^{\circledast}}(s)}t^{l_{\Lambda^*}(s)+1})} 
\end{equation}  
We will use induction on the number of columns
and on the fermionic degree of $\Lambda$. We first suppose that
the first column is of length $\ell$ and  does not contain a circle.  By Lemma~\ref{Relacionvcv}, we have 
\begin{align} \label{eqrec1}
\frac{v_\Lambda(q,t)}{v_{\mathcal{C}\Lambda}(q,t)} =\displaystyle\frac{t^{\binom{\ell}{2}}}{q^{\binom{m}{2}}}\cdot \frac{1}{\prod_{i=1}^{\ell}1-q^{\Lambda_i^{\circledast}-1}t^{\ell-(i-1)}}= \displaystyle\frac{t^{\binom{\ell}{2}}}{q^{\binom{m}{2}}}\cdot \frac{1}{\prod_s (1-q^{a_{\Lambda^{\circledast}}(s)}t^{l_{\Lambda^*}(s)+1})}
\end{align}
where the product is over the cells $s$ in the first column of $\Lambda$. 
By induction on the number of columns of $\Lambda$,
\eqref{vvaprouver}  will thus hold in that
case if we can show that
\begin{enumerate}
\item[(i)] 
$\displaystyle{
n(\S \Lambda)+n( (\Lambda')^a/\delta_m)=n(\S \mathcal C \Lambda)+n( (\mathcal C\Lambda')^a/\delta_m)+\binom{\ell}{2}
}$
\item[(ii)]
  $ \displaystyle{
(m-1) |\Lambda^a/\delta_m| - n(\Lambda^a/\delta_m)= (m-1) |\mathcal C \Lambda^a/\delta_m| - n(\mathcal C \Lambda^a/\delta_m)+\binom{m}{2}
}$
\end{enumerate}
This is indeed the case, since the first relation follows from
$$
n(\S \Lambda)-n(\S \mathcal C \Lambda)=\binom{\ell}{2} \qquad
{\rm and}
\qquad
n( (\Lambda')^a/\delta_m)=n( (\mathcal C\Lambda')^a/\delta_m)
$$
while the second is a consequence of
$$
|\Lambda^a/\delta_m| - |\mathcal C \Lambda^a/\delta_m|=m
\qquad
{\rm and}
\qquad
 n(\Lambda^a/\delta_m)- n(\mathcal C \Lambda^a/\delta_m)=  \binom{m}{2}
$$

We now consider the case where the first column of $\Lambda$ contains a circle.  In that case, we use Lemma~\ref{relacionvctv} to get (assuming that the length of $\Lambda$ is $\ell$) 
\begin{equation}  
\frac{v_\Lambda(q,t)}{v_{\mathcal{\tilde{C}}\Lambda}(q,t)}=\frac{t^{|(\Lambda')^a|-(m-1)^2}}{q^{|\Lambda^a|-(m-1)^2}}
\prod_{i \in fr(\tilde {\mathcal C} \Lambda)}(1-q^{\Lambda_i^{\circledast}-1}t^{\ell-1-(i-1)})
\end{equation}
By induction on the fermionic degree, \eqref{vvaprouver} will also hold in that case if we can prove that 
\begin{enumerate}
\item[(I)]  
$ \displaystyle{
  \prod_{s\in\mathcal{B}\Lambda} \frac{1}{1-q^{a_{\Lambda^{\circledast}}(s)}t^{l_{\Lambda^*}(s)+1}} = \frac{\prod_{i \in fr(\tilde {\mathcal C} \Lambda)}(1-q^{\Lambda_i^{\circledast}-1}t^{\ell-1-(i-1)})}{ \prod_{s\in\mathcal{B} \widetilde{\mathcal C}\Lambda}
    1-q^{a_{\Lambda^{\circledast}}(s)}t^{l_{\Lambda^*}(s)+1}}
}$
\item[(II)] 
$\displaystyle{
n(\S \Lambda)+n( (\Lambda')^a/\delta_m)=n(\S \widetilde{\mathcal C} \Lambda)+n( (\widetilde{\mathcal C}\Lambda')^a/\delta_{m-1})+|(\Lambda')^a|-(m-1)^2   \phantom{\binom{l}{2}}
} $
\item[(III)]
  $ \displaystyle{
(m-1) |\Lambda^a/\delta_m| - n(\Lambda^a/\delta_m)= (m-2) |\widetilde{\mathcal C} \Lambda^a/\delta_{m-1}| - n(\widetilde{\mathcal C} \Lambda^a/\delta_{m-1})+|\Lambda^a|-(m-1)^2 \phantom{\binom{l}{2}}
}$
\end{enumerate}
The first relation holds since for $s=(i,1)$
we have $a_{\Lambda^{\circledast}}(s)=\Lambda_i^{\circledast}-1$ and $l_{\Lambda^*}(s)=\ell-i-1$, which means that the contribution of the cells in fermionic rows in the first column of $\Lambda$ are canceled out.
Relation (II) is seen as follows: we have
$$
n(\S \Lambda)-n(\S \widetilde{\mathcal C} \Lambda)=n(\Lambda^\circledast)-n(\Lambda^*)-n(\delta_{m+1})+n(\delta_m)= \ell-1-\binom{m}{2}
$$
and
$$
n( (\Lambda')^a/\delta_m)-n( (\widetilde{\mathcal C}\Lambda')^a/\delta_m)=
n( (\Lambda')^a)-n( (\widetilde{\mathcal C}\Lambda')^a)-n(\delta_m)+n(\delta_{m-1})=|(\Lambda')^a|-(\ell-1) - \binom{m-1}{2}
$$
where the last relation follows from the fact that if $\hat \lambda=(\lambda_2,\lambda_3,\dots)$ (that is, if $\hat \lambda$ is the partition $\lambda$ without its first entry) then $n(\lambda)-n(\hat \lambda)=|\hat \lambda|=|\lambda|-\lambda_1$.  Relation (II) is then seen to hold since $\binom{m-1}{2}+\binom{m}{2}=(m-1)^2$.

Given that $\Lambda^a=(\widetilde{\mathcal C} \Lambda)^a$, we have that Relation (III) follows from
$
(m-1)|\Lambda^a|-(m-2)|\Lambda^a|=|\Lambda^a|
$
and
$$
(m-1)|\delta_m|-(m-2)|\delta_{m-1}|-n(\delta_m)+n(\delta_{m-1})=(m-1)\binom{m}{2}-(m-2)\binom{m-1}{2} -\binom{m-1}{2}=(m-1)^2
$$
\end{proof}

The following corollary will imply Theorem~\ref{theomain2}.
\begin{corollary} \label{coronouv}  We have that
\begin{equation} \label{vvaprouvernew}
\widetilde{v}_\Lambda(q,t) = \frac{t^{n(\widetilde{\S} \Lambda)+n( (\Lambda')^a/\delta_{m-1})}}
   {q^{(m-2) |\Lambda^a/\delta_{m-1}| - n(\Lambda^a/\delta_{m-1})}}
   \frac{1}{\prod_{s\in \mathcal{B}\Lambda} (1-q^{a_{\Lambda^{\circledast}}(s)}t^{l_{\Lambda^*}(s)+1})} 
  \end{equation}  
\end{corollary}
\begin{proof}  From Lemma~\ref{relacionE1E2} and \eqref{vvaprouver}, we obtain that
  \begin{equation}
    \widetilde{v}_\Lambda(q,t) = \frac{t^{n({\S} \Lambda)+n( (\Lambda')^a/\delta_{m})-|(\Lambda')^a|+(m-1)^2}}
   {q^{(m-1) |\Lambda^a/\delta_{m}| - n(\Lambda^a/\delta_{m})-|\Lambda^a|+(m-1)^2}}
   \frac{1}{\prod_{s\in \mathcal{B}\Lambda} (1-q^{a_{\Lambda^{\circledast}}(s)}t^{l_{\Lambda^*}(s)+1})} 
  \end{equation}  
  We therefore only need to show that the $q$ and $t$ powers are such
  as in \eqref{vvaprouvernew}.  Using
  $$
  n(\S \Lambda)-n(\widetilde{\S} \Lambda)= n(\Lambda^\circledast)-n(\Lambda^*)-
n(\delta_{m+1})+n(\delta_m)=
  \binom{m-1}{2}= |(\Lambda')^a|-\binom{m}{2} 
  $$
 the $t$ power is seen to be the correct one with the help of the relations 
  $n(\delta_{m})-n(\delta_{m-1})=\binom{m-1}{2}$ and
  $\binom{m}{2}+\binom{m-1}{2}=(m-1)^2$.  In order to prove that the $q$ power is such as in \eqref{vvaprouvernew}, it suffices to show that
  $$
-(m-1)|\delta_m|+(m-2)|\delta_{m-1}|+n(\delta_m)-n(\delta_{m-1})=-(m-1)^2
  $$
which can straightforwardly be checked using $|\delta_m|=\binom{m}{2}$ and
$n(\delta_{m})-n(\delta_{m-1})=\binom{m-1}{2}$.

\end{proof}

We now prove the claim made in Remark~\ref{remarkzeta} that
$\zeta_\Lambda$ and $n\bigl((\Lambda')^a/\delta_m\bigr)$ are equal.
\begin{lemma} \label{lemzeta}
Let $\Lambda=(\Lambda^a;\Lambda^s)$ be a superpartition of degree $m$. We have
\begin{align}
\zeta_\Lambda=n\bigl((\Lambda')^a\bigr)-n(\delta_m)=n\bigl((\Lambda')^a/\delta_m\bigr)
\end{align}
\end{lemma}
\begin{proof}
  For simplicity, we let $\Omega= \Lambda'$.  We thus have to prove that
  \begin{align} \label{zetaN}
    \zeta_{\Omega'}=n(\Omega^a)-n(\delta_m)
    \end{align}
  We will proceed by induction on the number of rows of $\Omega$.  Let $\hat \Omega$ be superpartition whose diagram is that of $\Omega$ without its first row.
  If the first row of the diagram of $\Omega$ does not contain a circle, then
 $$ 
   \zeta_{\hat \Omega'}= \zeta_{\Omega'} =n(\Omega^a)-n(\delta_m)=n(\hat \Omega^a)-n(\delta_m)
 $$ 
   and \eqref{zetaN}  holds by induction.  Otherwise, we have
$$
\zeta_{\Omega'}-\zeta_{\hat \Omega'}=    \bigl(\Omega_2^a-(m-1)\bigr)+\bigl(\Omega_3^a-(m-2)\bigr)+\cdots+\bigl(\Omega_m^a\bigr) =|\Omega^a|-\Omega_1^a -\binom{m-1}{2}
$$  
Hence  $\zeta_{\Omega'}-\zeta_{\hat \Omega'}= n(\Omega^a/\delta_m)-n(\hat \Omega^a/\delta_{m-1})$ since
$
n(\delta_m)-n(\delta_{m-1})=\binom{m-1}{2}
$
and, as we have seen in the proof of Theorem~\ref{theomain},
$$
n(\Omega^a)-n(\hat \Omega^a)= |\Omega^a|-\Omega^a_1
$$
Supposing by induction that $\zeta_{\hat \Omega'}= n(\hat \Omega^a/\delta_{m-1})$,
we thus have that \eqref{zetaN} holds again in that case.
\end{proof}

We can now give a combinatorial formula for the norm of a Macdonald polynomial in superspace.  The formula follows  from Theorem \ref{theomain} and Lemma ~\ref{lema1}. 
\begin{proposition} \label{propnorm}
For any superpartition $\Lambda$, we have
\begin{align} \label{eqnorm}
\|P_\Lambda\|^2=q^{|\Lambda^a|} \prod_{s\in \mathcal{B}\Lambda} \frac{1-q^{a_{\Lambda^*}(s)+1}t^{l_{\Lambda^\circledast}(s)}}{1-q^{a_{\Lambda^\circledast}(s)}t^{l_{\Lambda^*}(s)+1}}.
\end{align}

\end{proposition}

\begin{proof}
By Lemma~\ref{lema1}, we get that
\begin{align}\label{Norma}
  \|P_\Lambda\|^2 = \frac{1}{b_{\Lambda}(q,t)}= (-1)^{|\Lambda|+\binom{m}{2}} q^{|\Lambda|} \frac{\Phi_{\Lambda}^m(u(qt)^{-m};q,t)}{\Phi_{\Lambda'}^m(u;t^{-1},q^{-1})}= (-1)^{|\Lambda|+\binom{m}{2}} q^{|\Lambda|} \frac{\Phi_{\Lambda}^m(0;q,t)}{\Phi_{\Lambda'}^m(0;t^{-1},q^{-1})}
\end{align}
since $\|P_\Lambda\|^2$ does not depend on $u$.  Therefore, by Theorem \ref{theomain}, we have
\begin{equation}
\begin{split}
  \|P_\Lambda\|^2= (-1)^{|\Lambda|+\binom{m}{2}} q^{|\Lambda|} &  \frac{t^{-(m-1)|{(\Lambda')}^a/\delta_m|+2n({(\Lambda')}^a/\delta_m)+n(\mathcal{S}\Lambda)}}{q^{(m-1)|\Lambda^a/\delta_m|-2n(\Lambda^a/\delta_m)-n(\mathcal{S}\Lambda')}} \\
  & \qquad  \qquad  \times \prod_{s\in \mathcal{B}\Lambda}\frac{-1}{q^{a_{\Lambda^*}(s)+1}t^{l_{\Lambda^\circledast}(s)}}\cdot \frac{1-q^{a_{\Lambda^*}(s)+1}t^{l_{\Lambda^\circledast}(s)}}{1-q^{a_{\Lambda^{\circledast}}(s)}t^{l_{\Lambda^*}(s)+1}}
\end{split}
\end{equation}
given that $a_{\lambda'}(i,j)=l_{\lambda}(j,i)$ for any cell $(i,j)$ in a 
partition $\lambda$.   The number of cells in $\mathcal{B}\Lambda$ is equal to $|\Lambda|-\binom{m}{2}$, which means that the signs cancel out as wanted. It is also straightforward to show that
$$
(m-1)|\delta_m| -2n(\delta_m) -n (\delta_{m+1})=0
$$
using $|\delta_m|=\binom{m}{2}$ and
$$
n(\delta_m)= \sum_{k=1}^{m-1} \binom{k}{2}= \frac{m(m-1)(m-2)}{6} \, .
$$
To obtain \eqref{eqnorm}, we thus only have left to prove that
\begin{align} \label{final1}
  \sum_{s\in\mathcal{B}\Lambda} (a_{\Lambda^*}(s)+1)=n({\Lambda^\circledast}')+2n(\Lambda^a)-(m-1)|\Lambda^a| +|\Lambda| -|\Lambda^a|
\end{align}
and that
\begin{align} \label{final2}
\sum_{s\in\mathcal{B}\Lambda} l_{\Lambda^\circledast}(s)=n(\Lambda^\circledast)+2n((\Lambda')^a)-(m-1)|(\Lambda')^a| 
\end{align}
We will prove \eqref{final2} and then deduce \eqref{final1} from it.  First,  using
$\sum_{s \in \lambda} l_\lambda(s)=n(\lambda)$, we obtain
\begin{equation} \label{lll}
  \sum_{s\in\mathcal{B}\Lambda} l_{\Lambda^\circledast}(s)=n(\Lambda^\circledast) -
  \sum_{s\in\mathcal{F}\Lambda}  l_{\Lambda^\circledast}(s)
\end{equation}
where $\mathcal{F}\Lambda$ are the cells of $\Lambda$ that have a circle in both their row and
their column, that is, the cells in the diagram of $\Lambda$ that are not in $\mathcal{B}\Lambda$.  In order to make sense of $\sum_{s\in\mathcal{F}\Lambda}  l_{\Lambda^\circledast}(s)$, it is convenient to add the quantity $\zeta_\Lambda$ defined in Remark~\ref{remarkzeta} to get
\begin{align*}
  \sum_{s\in\mathcal{F}\Lambda}  l_{\Lambda^\circledast}(s) + \zeta_\Lambda & = \bigl[(m-1) (\Lambda')^a_1+(m-2) (\Lambda')^a_2+ \cdots + (\Lambda')^a_{m-1}\bigr] -n(\delta_m) \\
  & = (m-1)|(\Lambda')^a|-n((\Lambda')^a)-n(\delta_m)
\end{align*}  
From Lemma~\ref{lemzeta}, it is then immediate that  $\sum_{s\in\mathcal{F}\Lambda} l_{\Lambda^\circledast}(s)= (m-1)|(\Lambda')^a|-2n((\Lambda')^a)$.  We then see that \eqref{final2} follows from
\eqref{lll}.

Conjugating \eqref{final2}, we obtain that
\begin{equation} \label{lll2}
\sum_{s\in\mathcal{B}\Lambda} a_{\Lambda^\circledast}(s)=n({\Lambda^\circledast }')+2n(\Lambda^a)-(m-1)|\Lambda^a| 
\end{equation}  
Using $\sum_{s\in\mathcal{B}\Lambda} a_{\Lambda^\circledast}(s)= \sum_{s\in\mathcal{B}\Lambda} a_{\Lambda^*(s)}+|\Lambda^a|-\binom{m}{2}$, we can conclude that
$$
\sum_{s\in\mathcal{B}\Lambda} (a_{\Lambda^*(s)}+1)= \sum_{s\in\mathcal{B}\Lambda} a_{\Lambda^\circledast}(s)+|\Lambda|-|\Lambda^a|
$$
which, from \eqref{lll2}, implies \eqref{final1}. 
\end{proof}

\begin{example}
Let us compute the norm of the Macdonald superpolynomial $P_\Lambda$ where $\Lambda=(4,2,1;3,2)$. The contribution of each box in the Ferrers diagram of $\mathcal{B}\Lambda$ is given by\\
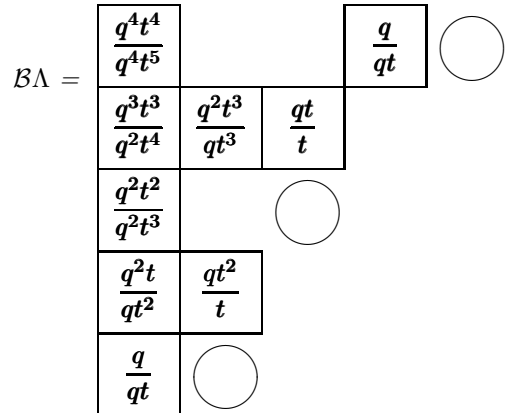
\begin{center}
$\mathcal{B}\Lambda\,\,$=\,\,\ytableausetup{boxsize=3em}
\begin{ytableau}
  {\displaystyle\pmb{\frac{q^4t^4}{q^4t^5}}}&\none & \none & {\displaystyle\pmb{\frac{q}{qt}}} &\none[\btcercle{}] \\       
   {\displaystyle\pmb{\frac{q^3t^3}{q^2t^4}}} & {\displaystyle\pmb{\frac{q^2t^3}{qt^3}}}  & {\displaystyle\pmb{\frac{qt}{t}}} \\
   {\displaystyle\pmb{\frac{q^2t^2}{q^2t^3}}} & \none &\none[\btcercle{}]\\
   {\displaystyle\pmb{\frac{q^2t}{qt^2}}} & {\displaystyle\pmb{\frac{qt^2}{t}}} \\
   {\displaystyle\pmb{\frac{q}{qt}}} &  \none[\btcercle{}]
\end{ytableau}
\end{center}
where $\pmb{\frac{a}{b}}=\frac{1-a}{1-b}$. Therefore the norm of $P_\Lambda$ is
\begin{align*}
\|P_\Lambda\|^2&=q^7 \cdot \frac{(1-q^4t^4)(1-q)^2(1-q^3t^3)(1-q^2t^3)(1-qt)(1-q^2t^2)(1-q^2t)(1-qt^2)}{(1-q^4t^5)(1-qt)^2(1-q^2t^4)(1-qt^3)(1-t)^2(1-q^2t^3)(1-qt^2)}\\
&=q^7\cdot\frac{(1-q^4t^4)(1-q)^2(1-q^3t^3)(1+qt)(1-q^2t)}{(1-q^4t^5)(1-q^2t^4)(1-qt^3)(1-t)^2}.
\end{align*}
\end{example}

\vspace{1cm}
\begin{acknow}  
This work was supported by FONDECYT (Fondo Nacional de Desarrollo Cient\'{\i}fico y Tecnol\'ogico de Chile) {doctoral grant \#{21150876}} (C. G.) and regular grant \#{1170924} (L. L.).
 \end{acknow}  

\begin{appendix} 
\section{Proofs of Propositions \ref{PropgI} and \ref{PropgII} } \label{appenB}
The proofs will mainly follow the steps of the corresponding proofs in the Jack polynomial in superspace case \cite{DLMeva}.  But in order to do so,
we first need to extend  to the Macdonald case a certain result of 
\cite{ForMcA} pertaining to Pieri-type rules for non-symmetric
Jack polynomials.

The non-symmetric Macdonald polynomials given in ~\eqref{nonS} have a non-homogeneous analog called the interpolation Macdonald polynomials \cite{knop}.
The interpolation Macdonald polynomial $E_\eta^*(x;q,t)$ is defined as
the unique polynomial of degree $\leq |\eta|$ whose coefficient in $x^\eta$ is equal to 1 and such that
\begin{align*}
E^*_\eta(\bar{\eta};q,t)\neq 0 \qquad  {\rm and} \qquad E^*_\eta(\bar{\mu};q,t)=0 \,\,{\rm ~for~~} |\mu|\leq |\eta| {\rm ~~and~~}\mu\neq \eta
\end{align*}
where  $\bar{\mu}=(\bar{\mu}_1,\ldots,\bar{\mu}_n)$, with $\bar{\mu}_i$ the eigenvalue defined in \eqref{eigenvalY}.  The non-symmetric Macdonald polynomials can be obtained by taking the highest degree term in an interpolation Macdonald polynomial, that is,
\begin{align}\label{desc1}
  E^*_\eta(x;q,t)=E_\eta(x;q^{-1},t^{-1})+\sum_{|\mu| < |\eta|} b_{\eta\mu}
  E_\mu(x;q^{-1},t^{-1}).
\end{align}
The interpolation Macdonald polynomials satisfy an extra vanishing condition that will prove crucial.  Let $\preccurlyeq$ be the partial order on compositions such that $\nu \preccurlyeq \eta$ if there exists a permutation $\pi$ such that $\nu_i < \eta_{\pi(i)}$ for $i < \pi(i)$ and $\nu_i \leq \eta_{\pi(i)}$ for
$i \geq \pi(i)$.  Then
\begin{align*}
 E^*_\eta(\bar{\mu};q,t)=0 \,\,{\rm ~if~~} \eta\nprec\mu
\end{align*}
This extra vanishing condition implies that
any analytic function $f(x)$ vanishing on $\{\bar{\mu}:\eta\nprec\mu\}$ can be written in the form 
\begin{align*}
f(x)=\sum_{\mu:\eta\preccurlyeq \mu} c_{\eta \mu}(q,t) E^*_{\mu}(x;q,t).
\end{align*}
We are interested in the case $f(x)=x_{i_1}\cdots x_{i_p} E_{\eta}^*(x;q,t)$
where $i_1,\dots,i_p$ are distinct elements of $\{1,\dots,N\}$.
Since it vanishes on $\{\bar{\mu}:\eta\nprec\mu\}$, we can write
\begin{align*}
x_{i_1}\cdots x_{i_p} E_{\eta}^*(x;q,t)=\sum_{\mu:\eta\preccurlyeq \mu} c_{\eta \mu}^{(i_1,\dots,i_p)}(q,t) E^*_{\mu}(x;q,t).
\end{align*}
From \eqref{desc1},  we can take the leading homogeneous term on both sides to get
\begin{align}
x_{i_1}\cdots x_{i_p} E_{\eta}(x;q^{-1},t^{-1})=\sum_{\mu:\eta\preccurlyeq \mu, |\mu|=|\eta|+p} c_{\eta \mu}^{(i_1,\ldots, i_p)}(q,t) E_{\mu}(x;q^{-1},t^{-1})
\end{align}
which is equivalent to
\begin{align}\label{Pieri1}
x_{i_1}\cdots x_{i_p} E_{\eta}(x;q,t)=\sum_{\mu:\eta\preccurlyeq \mu, |\mu|=|\eta|+p} c_{\eta \mu}^{(i_1,\ldots, i_p)}(q^{-1},t^{-1}) E_{\mu}(x;q,t)
\end{align}
We can obtain a better upper bound on the summation index $\mu$ using the orthogonality of the non-symmetric Macdonald polynomials.
Let $\text{C.T.}(f)$ denote the constant term of the Laurent series of
$f$ in the variables $x_1,\dots,x_N$.  Define the following scalar product on $\mathbb Q(q,t)[x_1,\dots,x_N]$:
\begin{equation}
\langle f, g \rangle_{q,t}:=\text{C.T.}\left\{ f(x;q,t)\, g(x^{-1};q^{-1},t^{-1})\,W(x;q,t)\right\},
\end{equation}
   where
\begin{equation}W(x;q,t)=\prod_{1\leq i<j \leq N}\frac{(x_i/x_j;q)_\y \,(qx_j/x_i;q)_{\y}}{(tx_i/x_j;q)_\y\,(qtx_j/x_i;q)_\y}.
\end{equation}
Note that this scalar product is sesquilinear, that is,
for $c=c(q,t) \in \mathbb Q(q,t)$, we have
\begin{equation}
\langle c \, f, g \rangle_{q,t} = c \, \langle f, g \rangle_{q,t} \quad
{\rm and} \quad \langle f, c \, g \rangle_{q,t} = \bar c \, 
\langle f, g \rangle_{q,t},
\end{equation}
where $\bar c=c(1/q,1/t)$.
It is known \cite{Mac,Mac1} that the  non-symmetric Macdonald polynomials $E_\eta(x;q,t)$ form an orthogonal set with respect to $\langle \cdot, \cdot \rangle_ {q,t}$.  From this orthogonality and the sesquilinearity of the scalar product, we have
\begin{align*}
  c_{\eta \mu}^{(i_1,\ldots, i_p)}(q,t) =\frac{\langle E_\mu \, ,
    \, x_{i_1}\cdots x_{i_p} E_\eta \rangle_{q ,t}}{\langle E_\mu \, , \, E_\mu \rangle_{q ,t}}.
\end{align*} 
 Since $ \langle x_k \, f,  x_k\, g \rangle_{q,t} = \langle   f,  g \rangle_{q,t}$ for all $k$ and $x_1 \cdots x_N E_\eta= E_{\eta+1^N}$, where $\eta+1^N$ is the composition $(\eta_1+1,\dots,\eta_N+1)$, we get
\begin{align*}
  c_{\eta \mu}^{(i_1,\ldots, i_p)}(q^{-1},t^{-1})& =\frac{\langle x_{j_1}\cdots x_{j_{N-p}}E_\mu  \, , \,  E_{\eta+1^N} \rangle_{q ,t}}{\langle E_\mu \, , \, E_\mu  \rangle_{q,t}}
   =c_{\mu\,\, \eta+1^N}^{(j_1,\ldots, j_{N-p})}(q,t)\frac{\langle  E_{\eta+(1^N)}\, ,\, E_{\eta+(1^N)}\rangle_{q ,t}}{\langle E_\mu \, , \, E_\mu \rangle_{q ,t}},
\end{align*}
where $\{j_1,\ldots, j_{N-p} \}$ is the complement of $\{i_1,\ldots, i_{p} \}$
in $\{1,\ldots, N \}$.
But using (\ref{Pieri1}) we have that
$c_{\eta\,\, \mu+1^N}^{(j_1,\ldots, j_{N-p})}(q^{-1},t^{-1})=0$
for $\mu \npreceq \eta+1^N$. Hence 
\begin{align*}
c_{\eta \mu}^{(i_1,\ldots, i_p)}(q^{-1},t^{-1})=0  \quad \mbox{for}  \quad \mu \npreceq \eta+1^N.
\end{align*}
Therefore we can make in (\ref{Pieri1}) the additional restriction $\mu \preccurlyeq \eta+(1^N)$ to obtain 
\begin{align}\label{Pieri2}
x_{i_1}\cdots x_{i_p} E_{\eta}(x;q,t)=\sum_{\mu\in \mathbb{J}_{N,p}} c_{\eta \mu}^{(i_1,\ldots, i_p)}(q^{-1},t^{-1}) E_{\mu}(x;q,t),
\end{align}
where
\begin{align*}
\mathbb{J}_{N,p}:=\{\mu:\eta\preccurlyeq \mu\preccurlyeq \eta+1^N, |\mu|=|\eta|+p\}.
\end{align*}
This result is the extension to the Macdonald case of a result of \cite{ForMcA} (see also \cite{Ba} where the Macdonald case was considered for $e_r(x)$ instead of for $x_{i_1}\cdots x_{i_p}$).  Given $(i_1,\dots,i_p)$, a more explicit description of the set 
$\mathbb{J}_{N,p}$ is derived in \cite{ForMcA} which can be interpreted as all the rearrangements of the rows of the new diagram such that the rows with a cell added can only move downwards or stay stationary, while the remaining rows can only move upwards or stay stationary.  This interpretation will play a fundamental role in the proofs of Propositions  \ref{PropgI} and  \ref{PropgII}.

\begin{example}
Consider the composition $\eta=(5,2,0,1)$ and its  diagram
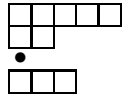
\begin{center}
\ysdiagram{{}{}{}{}{},{}{},{\none[\bullet]},{}{}{}} \,\,\ysdiagram{}
\end{center}
Suppose $p=2$, with $i_1=2$ and $i_2=3$.  Then the set $\mathbb{J}_{N,p}$ consists of the compositions whose diagrams are
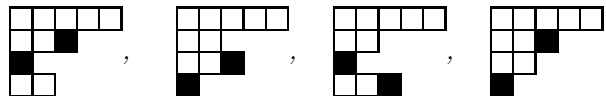
\begin{center}
\ysdiagram{{}{}{}{}{},{}{}{*(black)},{*(black)},{}{}}, \quad \ysdiagram{{}{}{}{}{},{}{},{}{}{*(black)},{*(black)}},\quad\ysdiagram{{}{}{}{}{},{}{},{*(black)},{}{}{*(black)}},\quad \ysdiagram{{}{}{}{}{},{}{}{*(black)},{}{},{*(black)}} 
\end{center}

\end{example}


\noindent{\it Proof of Proposition \ref{PropgI}:}
Using the notation of Section~\ref{sps}, we let
\begin{align} \label{osym}
\mathcal{O}_{sym}=
\sum_{\sigma \in S_N/(S_m\times S_{m^c})}\mathcal{K}_\sigma \theta_1\cdots\theta_m A_{m}U^+_{m^c}.
\end{align} 
Using \eqref{UvsA} with $N$ replaced by $m$, we get
\begin{align*}
  A_m f=t^{\binom{m}{2}}\frac{\Delta_m}{\Delta_{m}^t} U_m^-f
\end{align*}
which immediately implies from \eqref{Urel}
that  $\mathcal{O}_{sym}T_i =t \, \mathcal{O}_{sym}$ if $i=m+1,\ldots,N-1$ and $\mathcal{O}_{sym}T_i =-\mathcal{O}_{sym} $ if $i=1,\ldots,m-1$. Hence, \eqref{property2} gives
\begin{align*}
 P_\Lambda\propto \mathcal{O}_{sym} E_\eta,
\end{align*}
whenever the first $m$ components of $\eta$ are a rearrangement of $\Lambda_1,\ldots,\Lambda_m$ and the last $N-m$ components of $\eta$ are a rearrangement of $\Lambda_{m+1},\ldots,\Lambda_N$ (the symbol $\propto $ means that the result holds up to a constant). 
We can now prove that
\begin{align*}
e_r \, P_\Lambda=\sum_\Omega g_{\Lambda,1^n}^\Omega P_\Omega
\end{align*}
where $g_{\Lambda,1^n}^\Omega\neq 0$ only if $\Omega/\Lambda$ is a vertical $n$-strip, i.e. only if $\Omega^*/\Lambda^*$ and $\Omega^\circledast/\Lambda^\circledast$ are vertical $n$-strips. First, we observe that
\begin{align*}
e_r P_\Lambda  \propto  e_r \mathcal{O}_{sym}   E_{\tilde{\Lambda}} = \mathcal{O}_{sym} e_r  E_{\tilde{\Lambda}}=\mathcal{O}_{sym} \sum_{i_1<\ldots <i_r} x_{i_1}\ldots x_{i_r} E_{\tilde{\Lambda}}
\end{align*}
since $e_r$ is symmetric in $x_1,\dots,x_N$. Given a composition $\eta$ whose first $m$ entries are distinct, let $\Omega_\eta=(\Omega_\eta^a;\Omega_\eta^s)$ be the superpartition such that $\Omega_\eta^a$ (resp. $\Omega_\eta^s$) is obtained by rearranging the first $m$ entries of $\eta$ (resp. the last $N-m$ entries of $\eta$). Note that we will never consider the case where the first $m$ entries of $\eta$ are not all distinct since in this case $\mathcal{O}_{sym} E_{\eta}=0$ given that
$A_m K_{i,i+1}=0$ and
$K_{i,i+1} E_\eta= E_\eta$ if $\eta_i=\eta_{i+1}$ by \eqref{property2}.
We thus need to show that the $\eta$'s that appear in $x_{i_1}\cdots x_{i_r}E_{\tilde{\Lambda}}$ are all such that $\Omega_\eta/{\Lambda}$ is a vertical $r$-strip 

We have that \eqref{Pieri2} says in particular that $\eta$ is obtained by adding $r$ boxes in distinct rows of $\tilde \Lambda$.  It is thus immediate that $\Omega^*_\eta/{\Lambda^*}$ is a vertical $r$-strip. To see that  $\Omega_\eta^\circledast/{\Lambda}^\circledast$ is also a vertical $r$-strip, we use the interpretation of $\mathbb{J}_{N,r}$ given in the paragraph that follows \eqref{Pieri2} which tells us that the rows with a cell added can only stay in their position or move downwards in the corresponding diagrams.  Since the rows of $\eta$ are a rearrangement of those of $\tilde \Lambda$ plus a certain vertical $r$-strip, the only way $\Omega_\eta^\circledast/{\Lambda}^\circledast$ would not be a vertical $r$-strip is if one of the last
$N-m$ rows of
$\tilde \Lambda$ gained a cell and then was rearranged into one of the first $m$ rows (thus gaining two cells when considered in $\Omega^\circledast$). But this is impossible from the previous argument since the last $N-m$ rows of
$\tilde \Lambda$
cannot move to one of the first $m$ rows  if they gained a cell (since in this case they would have moved up in the diagram).

\begin{flushright}
$\square$
\end{flushright}

 \noindent {\it Proof of Proposition \ref{PropgII}:}
We have to prove that
\begin{align*}
\tilde{e}_r P_\Lambda=\sum_{\Omega} g_{\Lambda,(0;1^n)}^\Omega P_\Omega
\end{align*}
where $g_{\Lambda,(0;1^n)}^\Omega\neq 0$ only if $\Omega/\Lambda$ is a vertical $\tilde{n}$-strip, that is, only if $\Omega^*/\Lambda^*$ is a vertical $n$-strip
and $\Omega^\circledast/\Lambda^\circledast$ is a vertical $(n+1)$-strip.

It was established in the proof of Lemma~4 in \cite{JoLap} that
\begin{align*}
\tilde{e}_r \mathcal{O}_{sym}|_{\theta_1\cdots\theta_{m+1}}&=(-1)^m A_{m+1} e_{r}^{(m+1)} U_{m^c}^+,
\end{align*}
where $e_r^{(m+1)}$ is obtained by letting $x_{m+1}=0$ in $e_r$, and where
$f|_{\theta_1 \cdots \theta_{m+1}}$ stands for the coefficient of $\theta_1 \cdots \theta_{m+1}$ in $f$.  We then use the well-known 
factorization 
\begin{align*}
U_{m^c}^+=U_{(m+1)^c}^+ \bigl(1+T_{m+1}+T_{m+1}T_{m+2}+T_{m+1}T_{m+2}T_{m+3}\cdots+T_{m+1}T_{m+2}\cdots T_{N-1} \bigr)  
\end{align*}
and the  commutation of $e_r^{(m+1)}$ and $U_{(m+1)^c}^+$  to deduce that
\begin{align}
\tilde{e}_r P_\Lambda|_{\theta_1\cdots\theta_{m+1}} \propto A_{m+1} e_{r}^{(m+1)} U_{(m+1)^c}^+ L_m \,  E_{\tilde \Lambda} = A_{m+1} U_{(m+1)^c}^+  e_{r}^{(m+1)} L_m \,  E_{\tilde \Lambda} 
\end{align}  
where $L_m=1+T_{m+1}+T_{m+1}T_{m+2}+T_{m+1}T_{m+2}T_{m+3}\cdots+T_{m+1}T_{m+2}\cdots T_{N-1}$.  We thus have that
\begin{align}
\tilde{e}_r P_\Lambda \propto \mathcal{O}_{sym}^{(m+1)}  e_{r}^{(m+1)} L_m \,  E_{\tilde \Lambda} 
  \end{align}  
where the superscript $m+1$ in $\mathcal{O}_{sym}^{(m+1)}$ indicates that $m$
is replaced by $m+1$ in \eqref{osym}.   From \eqref{property2},
we observe that $L_m E_{\tilde{\Lambda}}$ is a sum of $E_\nu$'s where $\nu$ is obtained by rearranging the last  $N-m$ rows of $\tilde{\Lambda}$.  Now
$e_r^{(m+1)} E_\nu$ is  a sum of terms  of the type
\begin{align}\label{Pieri4}
x_{i_1}\cdots x_{i_r}E_\nu 
\end{align}
where $m+1\notin \{i_1,\ldots,i_r\}$ and where we recall that $\nu$ is obtained by rearranging the last $N-m$ rows of $\tilde \Lambda$.  
Given a composition $\eta$ whose first $m+1$ entries are distinct, we let
$\tilde \Omega_\eta=(\tilde \Omega_\eta^a;\tilde \Omega_\eta^s)$ be the superpartition such that $\tilde \Omega_\eta^a$ (resp. $\tilde \Omega_\eta^s$) is obtained by rearranging the first $m+1$ entries of $\eta$ (resp. the last $N-m-1$ entries of $\eta$).  As in the proof of Proposition~\ref{PropgI}, it is immediate from
\eqref{Pieri2} that $\tilde \Omega^*_\eta/{\Lambda^*}$ is a vertical $r$-strip.
To show that  $\tilde \Omega_\eta^\circledast/{\Lambda}^\circledast$ is a vertical $(r+1)$-strip turns out however to be slightly more subtle than to show the corresponding case (that
$\Omega_\eta^\circledast/{\Lambda}^\circledast$ is a vertical $r$-strip) in Proposition~\ref{PropgI}.  Since $\nu$ is simply a rearrangement of the last $N-m$ rows of $\tilde \Lambda$, we can consider for simplicity  that
$\nu$ is equal to $\tilde \Lambda$.  Now $e_r^{(m+1)}$ does not add a cell in row $(m+1)$ of $\tilde \Lambda$.  By  the interpretation of $\mathbb{J}_{N,r}$ given in the paragraph that follows \eqref{Pieri2}, row $m+1$ can only stay stationary or move upwards in the corresponding diagrams (which means that $\tilde \Lambda_{m+1}$ will gain exactly one cell when considered in $\tilde \Omega^\circledast$).  The rest of the proof is more or less  exactly as in the proof that
$\Omega_\eta^\circledast/{\Lambda}^\circledast$ is a vertical $r$-strip in Proposition~\ref{PropgI}.  Since the remaining rows of $\eta$ are a rearrangement of those of $\tilde \Lambda$ plus a certain vertical $r$-strip, the only way $\Omega_\eta^\circledast/{\Lambda}^\circledast$ would not be vertical $r$-strip is if one of the last
$N-m$ rows of
$\tilde \Lambda$ gained a cell and then was rearranged into one of the first $m+1$ rows (thus gaining two cells when considered in $\Omega^\circledast$). But this is impossible since the last $N-m-1$ rows of
$\tilde \Lambda$
cannot move to one of the first $m+1$ rows  if they gained a cell (since in this case they would have moved upwards in the diagram). 
\begin{flushright}
$\square$
\end{flushright}

\end{appendix}

\end{document}